\documentclass[reqno]{amsart}
\usepackage{amscd,amsfonts,amssymb}
\usepackage{graphicx}
\usepackage{color}

\numberwithin{equation}{section}

\newtheorem{Theorem}{Theorem}[section]
\newtheorem{Lemma}{Lemma}[section]

\newtheorem{Proposition}{Proposition}[section]

\theoremstyle{definition}

\theoremstyle{remark}
\newtheorem{Remark}{Remark}[section]

\renewcommand{\r}{\rho}

\renewcommand{\u}{{\bf u}}
\renewcommand{\H}{{\bf H}}
\newcommand{\R}{{\mathbb R}}
\newcommand{\Dv}{{\rm div}}

\newcommand{\m}{{\bf m}}

\newcommand{\na}{\nabla}
\newcommand{\dl}{\delta}

\def\f{\frac}
\renewcommand{\O}{\Omega}
\def\ov{\overline}

\def\hf1{^\f{1}{1-\xi^2}}

\def\be{\begin{equation}}
\def\en{\end{equation}}
\def\bs{\begin{split}}
\def\es{\end{split}}

\newcommand{\ess}{{\rm ess\sup}}

\def\Xint#1{\mathchoice
   {\XXint\displaystyle\textstyle{#1}}%
   {\XXint\textstyle\scriptstyle{#1}}%
   {\XXint\scriptstyle\scriptscriptstyle{#1}}%
   {\XXint\scriptscriptstyle\scriptscriptstyle{#1}}%
   \!\int}
\def\XXint#1#2#3{{\setbox0=\hbox{$#1{#2#3}{\int}$}
    \vcenter{\hbox{$#2#3$}}\kern-.5\wd0}}
\def\dashint{\Xint-}

\author{Xianpeng Hu \and Dehua Wang  }
\address{Department of Mathematics, University of Pittsburgh,
                           Pittsburgh, PA 15260, USA.}
\email{xih15@pitt.edu}
\address{Department of Mathematics, University of Pittsburgh,
                           Pittsburgh, PA 15260, USA.}
\email{dwang@math.pitt.edu}

\title[Global Solutions of  Magnetohydrodynamic Equations]
{Global Existence and Large-Time Behavior of Solutions to the Three-Dimensional Equations of Compressible Magnetohydrodynamic Flows}

\keywords{Three-dimensional magnetohydrodynamics (MHD) equations, global solutions, large-time behavior, weak convergence, renormalized solutions.}

\subjclass{35Q36, 35D05, 76W05.}

\begin{document}

\begin{abstract}
The three-dimensional equations of compressible magnetohydrodynamic isentropic flows
are considered. An initial-boundary value problem is studied in a bounded domain with large data. The existence and large-time behavior of global weak solutions are established through a three-level approximation, energy estimates, and weak convergence for the adiabatic  exponent $\gamma>\frac32$ and constant viscosity coefficients.
\end{abstract}

\maketitle

\section{Introduction}  \label{sec1}

%Magnetohydrodynamics studies the interaction of electrically conducting fluids with %magnetic fields. The fluids can be ionized gases (commonly called plasmas) or liquid %metals. Magnetohydrodynamic (MHD) phenomena occur naturally in the Earth's interior, %constituting the dynamo that produces the Earth's magnetic field; in the magnetosphere that %surrounds the Earth; and in the Sun and throughout astrophysics. In the laboratory, %magnetohydrodynamics is important in the magnetic confinement of plasmas in experiments on %controlled thermonuclear fusion. Magnetohydrodynamic principles are also used in plasma %accelerators for ion thrusters for spacecraft propulsion, for light-ion-beam powered %inertial confinement, and for magnetohydrodynamic power generation.

Magnetohydrodynamics (MHD) concerns the motion of conducting fluids
in an electromagnetic field with a very broad range of applications.
The dynamic motion of the fluid and the magnetic field interact
strongly on each other. The hydrodynamic and
electrodynamic effects are coupled. The equations of three-dimensional
compressible magnetohydrodynamic flows in the isentropic case have the following form (\cite{Ca, KL, LL}):
\begin{equation} \label{e1}
\begin{cases}
\rho_t +\Dv(\r\u)=0, \\
(\r\u)_t+\Dv\left(\r\u\otimes\u\right)+\nabla p
  =(\na \times \H)\times \H+\mu \Delta \u+(\lambda+\mu)\nabla(\Dv\u), \\
\H_t-\nabla\times(\u\times\H)=-\nabla\times(\nu\nabla\times\H),\quad
\Dv\H=0,
\end{cases}
\end{equation}
where $\r$ denotes the density, $\u\in\R^3$ the velocity,
$\H\in\R^3$ the magnetic field, $p(\r)=a\r^\gamma$ the pressure with
constant $a>0$ and the adiabatic exponent $\gamma>1$; the viscosity coefficients of the flow satisfy $2\mu+3\lambda>0$ and $\mu>0$; $\nu>0$ is the magnetic
diffusivity acting as a magnetic diffusion coefficient of the
magnetic field, and all these kinetic coefficients and the magnetic
diffusivity are independent of the magnitude and direction of the
magnetic field. The symbol $\otimes$ denotes the Kronecker tensor product. Usually, we refer to the first equation in \eqref{e1} as the continuity equation, and the second equation as the momentum balance equation. It is well-known that
the electromagnetic fields are governed by the Maxwell's equations.
In magnetohydrodynamics, the displacement current  can be neglected (\cite{KL, LL}).
As a consequence, the last equation in \eqref{e1} is called the induction equation,
and  the electric field can be written in terms of the magnetic
field $\H$ and the velocity $\u$,
\begin{equation*} %\label{elec}
{\bf E}=\nu\nabla\times\H - \u\times\H.
\end{equation*}
Although the electric field ${\bf E}$ does not appear in \eqref{e1},
it is indeed induced according to the above relation by the moving conductive
flow in the magnetic field.

In this paper, we are interested in the global existence and large-time behavior of solutions to the three-dimensional MHD equations \eqref{e1} in a bounded domain $\O\subset \R^3$ with the following initial-boundary conditions:
\begin{equation}\label{e4}
\begin{cases}
\r(x,0)=\r_0(x)\in L^\gamma(\O), \quad \r_0(x)\geq 0,\\
\r(x,0)\u(x,0)=\m_0(x)\in L^1(\O),  \; \m_0=0 \textrm{ if } \r_0=0, \;  \f{|\m_0|^2}{\r_0}\in L^1(\O),\\
\H(x,0)=\H_0(x)\in L^2(\O), \quad \Dv\H_0=0 \textrm{ in } \mathcal{D}'(\O),\\
\u|_{\partial\O}=0, \quad \H|_{\partial\O}=0.
\end{cases}
\end{equation}
There have been a lot of studies on MHD by physicists and
mathematicians because of its physical importance, complexity, rich
phenomena, and mathematical challenges; see \cite{gw, gw2, f3, FJN,
h1, HT, LL, w1} and the references cited therein. In particular, the
one-dimensional problem has been studied in many papers, for
examples, \cite{gw, gw2, FJN,HT,sm,tz,w1} and so on. However,
many fundamental problems for MHD are still open. For example, even
for the one-dimensional case, the global existence of classical
solution to the full perfect MHD equations with large data remains
unsolved when all the viscosity, heat conductivity, and diffusivity
coefficients are constant, although the corresponding problem for
the Navier-Stokes equations was solved in \cite{KS} long time ago.
The reason is that the presence of the magnetic field and its
interaction with the hydrodynamic motion in the MHD flow of large
oscillation cause serious difficulties. In this paper we consider
the global weak solution to the three-dimensional MHD problem with
large data, and investigate the fundamental problems such as global
existence and large-time behavior. A multi-dimensional
nonisentropic MHD system for gaseous stars coupled with the Poisson
equation is studied in \cite{f3}, where all the viscosity
coefficients depend on temperature, and the pressure depends on
density asymptotically like the isentropic case
$p(\r)=a\r^\f{5}{3}$. In this paper, we study the multi-dimensional
isentropic problem \eqref{e1}-\eqref{e4} with $\gamma>\frac32$,
where all the viscosity coefficients $\mu, \lambda, \nu$ are
constant. We remark that $\gamma=\frac53$ for the monoatomic gases.

When there is no electromagnetic field, system \eqref{e1} reduces to the compressible Navier-Stokes equations. See \cite{f2, hoff97,p2} and their references for the studies on
the multi-dimensional Navier-Stokes equations. In particular, to overcome the difficulties of large oscillations of solutions, especially of density,  the concept of a renormalized solutions
is used in \cite{p2,f2}. Based on this idea, we study the initial-boundary value problem \eqref{e1}-\eqref{e4} for the MHD system in a bounded three-dimensional domain $\O$.
The goal of this paper is to
establish the existence of global weak solutions for large initial data
in certain functional spaces for $\gamma>\f{3}{2}$ and to study the large-time behavior of global weak solutions when the magnetic field and interaction present.
The existence of global weak solutions is proved by  using the Faedo-Galerkin method and the vanishing viscosity
method. We first obtain \textit{a priori} estimates directly from \eqref{e1}, which is the backbone of our result. In the proof of
the existence, we use the similar approximation scheme to that in \cite{f1} which consists of Faedo-Galerkin approximation, artificial viscosity, and artificial pressure. Then, motivated by the work in \cite{r1},  we show that an improvement on the
integrability of density can ensure the effectiveness and convergence of our approximation scheme. More specifically, we show that the uniform bound
of $\r^\gamma\textrm{ln}(1+\r)$ in $L^1$,  rather than the uniform bound of $\r^{\gamma+\theta}$ in $L^1$ for some $\theta>0$ as used in \cite{f1, f2, p2}, ensures the vanishing of artificial pressure and the strong convergence of the density.
To overcome the difficulty arising from the possible large oscillations of the density
$\r$, we adopt the method in Lions \cite{p2} and Feireisl \cite{f2} which is based on the celebrated
weak continuity of the effective viscous flux $p-(\lambda+2\mu)\Dv\u$ (see also Hoff \cite{hoff95}).
The estimates obtained by our approach produce further the large-time behavior of the global weak solutions to the initial-boundary value problem \eqref{e1}-\eqref{e4}. To achieve our goal for the MHD problem, we also need to develop estimates to deal with the magnetic field and its coupling and interaction with the fluid variables.
The nonlinear term $(\nabla\times\H)\times\H$ will be dealt with by the idea arising in incompressible Navier-Stokes equations.

We organize the rest of the paper as follows. In Section 2, we derive \textit{a priori} estimates from \eqref{e1},  give the definition of the weak solutions, and  also state our main results and give our approximation scheme. In Section 3, we will show the
unique solvability of the magnetic field in terms of the velocity field. In Section 4, following the method in \cite{f2} and the result obtained in Section 3, we show the existence of solutions to the approximation system.
In Section 5, we follow the technique in \cite{f1} with some modifications to get the strong convergence of $\r$ in $L^1((0,T)\times\O)$. In Section 6, motivated by \cite{f5} we  study the large-time behavior of global weak solutions to \eqref{e1}-\eqref{e4}.

\bigskip

\section{Main Results}

In this section, we reformulate the initial-boundary value problem \eqref{e1}-\eqref{e4} and state the main results.

%At the beginning, we derive \textit{a priori} estimates for MHD systems.
We first formally derive the energy equation and some \textit{a priori} estimates.
Multiplying the second equation in \eqref{e1} by $\u$, integrating over
$\O$, and using the boundary condition in \eqref{e4}, we obtain
\begin{equation}\label{f1}
\begin{split}
&\f{d}{dt}\int_{\O}\left(\f{1}{2}\r\u^2+\f{a}{\gamma-1}\r^\gamma\right)\mathrm{d}x
  +\int_{\O}\left(\mu|D\u|^2+(\lambda+\mu)(\Dv\u)^2\right)\mathrm{d}x \\
&=\int_{\O}\left((\na \times \H)\times \H\right)\cdot\u\;\mathrm{d}x.
\end{split}
\end{equation}
The term on the right hand side of \eqref{f1} can be rewritten as
$$\int_{\O}((\na \times \H)\times\H)\cdot\u\;\mathrm{d}x
=-\int_{\O}\left(\H^\top\na\u\,\H+\f{1}{2}\na(|\H|^2)\cdot\u\right)\mathrm{d}x.$$
Hence,  \eqref{f1} becomes
\begin{equation}\label{f2}
\begin{split}
&\f{d}{dt}\int_{\O}\left(\f{1}{2}\r\u^2+\f{a}{\gamma-1}\r^\gamma\right)\mathrm{d}x
  +\int_{\O}\left(\mu|D\u|^2+(\lambda+\mu)(\Dv\u)^2\right)\mathrm{d}x \\
&=-\int_{\O}\left(\H^\top\na\u\,\H+\f{1}{2}\na(|\H|^2)\cdot\u\right)\mathrm{d}x.
\end{split}
\end{equation}
Multiplying the third equation in \eqref{e1} by $\H$,
integrating over $\O$, and using the boundary condition in \eqref{e4} and the %solenoidal
condition $\Dv\H=0$, one has
\begin{equation}\label{f3}
\f{1}{2}\f{d}{dt}\int_{\O}|\H|^2\mathrm{d}x
+\int_{\O}(\nabla\times(\nu\nabla\times\H))\cdot\H\;\mathrm{d}x
=\int_{\O}(\nabla\times(\u\times\H))\cdot\H\;\mathrm{d}x.
\end{equation}
Direct calculations show that
$$\int_{\O}(\nabla\times(\nu\nabla\times\H))\cdot\H\;\mathrm{d}x
=\nu\int_{\O}|\nabla\times\H|^2\mathrm{d}x,$$
$$\int_{\O}(\nabla\times(\u\times\H))\cdot\H\;\mathrm{d}x
=\int_{\O}\left(\H^\top\na\u\,\H+\f{1}{2}\na(|\H|^2)\cdot\u\right)\mathrm{d}x.$$
Thus \eqref{f3} yields
\begin{equation}\label{f4}
\f{1}{2}\f{d}{dt}\int_{\O}|\H|^2\mathrm{d}x
+\nu\int_{\O}|\nabla\times\H|^2\mathrm{d}x
=\int_{\O}\left(\H^\top\na\u\,\H+\f{1}{2}\na(|\H|^2)\cdot\u\right)\mathrm{d}x.
\end{equation}
Adding \eqref{f2} and \eqref{f4} gives
\begin{equation}\label{f5}
\begin{split}
&\f{d}{dt}\int_{\O}
\left(\f{1}{2}\r\u^2+\f{a}{\gamma-1}\r^\gamma+\f{1}{2}|\H|^2\right)\mathrm{d}x\\
&+\int_{\O}\left(\mu|D\u|^2+(\lambda+\mu)(\Dv\u)^2
 +\nu |\nabla\times\H|^2\right)\mathrm{d}x=0.
\end{split}
\end{equation}

From Lemma 3.3 in \cite{pm1}, our assumptions on initial
data, and \eqref{f5}, we have the following \textit{a priori} estimates:
\begin{equation*}
\begin{split}
&\r|\u|^2\in L^{\infty}([0,T]; L^1(\O));\\
&\u\in L^2([0,T]; H^1_0(\O)); \\
&\H\in L^2([0,T]; H^1_0(\O))\cap L^{\infty}([0,T]; L^2(\O)); \\
&\r\in L^{\infty}([0,T]; L^{\gamma}({\O})).
\end{split}
\end{equation*}

Multiplying the continuity equation (i.e., the first equation in \eqref{e1})
by $b'(\rho)$,  we obtain the renormalized continuity equation:
\begin{equation}\label{e5}
b(\rho)_t+\Dv(b(\r)\u)+(b'(\r)\r-b(\rho))\Dv \u=0,
\end{equation}
for some suitable function $b\in C^1(\R^+)$.
Following the strategy in \cite{p2, f2}, we introduce the concept of \textit{finite energy weak solution}  $(\r, \u, \H)$ to the initial-boundary value problem \eqref{e1}-\eqref{e4}
in the following sense:

\begin{itemize}%\addtolength{\itemsep}{-0.5\baselineskip}
\item The density $\r$ is a non-negative function,
$$\r\in C([0,T];L^1(\O))\cap L^\infty([0,T]; L^\gamma(\O)), \quad \r(x,0)=\r_0,$$
and the momentum $\r\u$ satisfies
$$\r\u\in C([0,T]; L_{weak}^{\f{2\gamma}{\gamma+1}}(\O));$$

\item The velocity $\u$ and the magnetic field $\H$ satisfy the following:
$$\u\in L^2([0,T]; H_0^1(\O)), \quad
\H\in L^2([0,T]; H_0^1(\O))\cap C([0,T]; L_{weak}^2(\O)),$$
$\r\u\otimes\u$, $\na\times(\u\times\H)$, and $(\na \times \H)\times \H$ are integrable on
$(0,T)\times\O$, and
$$\r\u(x,0)=\m_0, \quad \H(x,0)=\H_0, \quad
\Dv\H=0 \textrm{ in }\mathcal{D}'(\O);$$

\item The system \eqref{e1} is satisfied in
$\mathcal{D}'(\R^3\times (0,T))$ provided that $\r$, $\u$, and $\H$ are prolonged to be zero outside $\O$;

\item The continuity equation in \eqref{e1} is satisfied in the sense of renormalized
solutions, that is, \eqref{e5} holds in $\mathcal{D}'(\O\times(0,T))$ for any $b \in C^1(\R^+)$
satisfying
\begin{equation}\label{x2}
b'(z)=0 \textrm{ for all $z\in \R^+$ large enough, say, } z\geq z_0,
\end{equation}
where the constant $z_0$ depends on the choice of  function b;

\item The energy inequality
$$E(t)
+\int_0^t\int_{\O}\left(\mu|D\u|^2+(\lambda+\mu)(\Dv\u)^2+\nu|\nabla\times\H|^2\right)
\mathrm{d}x\mathrm{d}s\leq E(0),$$
holds for a.e $t\in[0,T]$,
where
$$E(t)=\int_{\O}\left(\f{1}{2}\r\u^2+\f{a}{\gamma-1}\r^\gamma+\f{1}{2}|\H|^2\right)\mathrm{d}x,$$
and
$$E(0)=\int_{\O}\left(\f{1}{2}\f{|\m_0|^2}{\r_0}+\f{a}{\gamma-1}\r_0^\gamma+\f{1}{2}|\H_0|^2
\right)\mathrm{d}x.$$
\end{itemize}

\begin{Remark} As a matter of fact, the function $b$ does not need to be bounded. By Lebesgue Dominated convergence theorem,
we can show that if $\r, \u$ is a pair of finite
energy weak solutions in the renormalized sense, they also satisfy \eqref{e5} for any $b\in C^1(0, \infty)\cap C[0, \infty)$
satisfying
\begin{equation}\label{e12}
|b'(z)z| \leq cz^{\f{\gamma}{2}} \textrm{ for z larger than some
positive constant } z_0.
\end{equation}
\end{Remark}

Now our main result on the existence of {finite energy weak solutions} reads as follows.

\begin{Theorem}\label{mt}
Assume that $\O\subset R^3$ be a bounded domain with a boundary of class $C^{2+\kappa}$, $\kappa>0$, and $\gamma>\f{3}{2}$. Then for any given $T>0$, the initial-boundary value problem \eqref{e1}-\eqref{e4} has a finite energy weak solution $(\r, \u, \H)$
on  $\O\times (0,T)$.
\end{Theorem}

\begin{Remark}
The fluid density $\r$ as well as the momentum $\r\u$ should be recognized in the sense of instantaneous values
(cf. Definition 2.1 in \cite{f2}) for any time $t\in[0,T]$.
\end{Remark}

%\begin{Remark}
%In comparison with the work in \cite{f3}, our result enlarges the value of the adiabatic %constant $\gamma$
%and derives \textit{a priori} estimate directly from the system \eqref{e1}-\eqref{e4}.
%\end{Remark}

As a direct application of Theorem \ref{mt}, we have the following result on the large-time behavior of solutions to the problem \eqref{e1}-\eqref{e4}:

\begin{Theorem}\label{mt2}
Assume that $(\r, \u, \H)$ is the finite energy weak solution to \eqref{e1}-\eqref{e4} obtained in Theorem \ref{mt}, then there exist a stationary state
of density $\r_s$ which is a positive constant, a stationary state of velocity $\u_s=0$, and a stationary state of  magnetic field $\H_s=0$ such that, as $t\to\infty$,
\begin{equation}\label{xx1}
\begin{cases}
\r(x,t)\rightarrow \r_s \textrm{ strongly in } L^\gamma(\O);\\
\u(x,t)\rightarrow \u_s=0\textrm{ strongly in }L^2(\O);\\
\H(x,t)\rightarrow \H_s=0\textrm{ strongly in }L^2(\O).
\end{cases}
\end{equation}
\end{Theorem}

The proof of Theorem \ref{mt} is based on the following approximation problem:
\begin{equation}\label{e7}
\begin{cases}
\rho_t +\Dv(\r\u)=\varepsilon\Delta\r, \\
(\r\u)_t+\Dv\left(\r\u\otimes\u\right)+a\nabla \r^\gamma+\delta\nabla \r^\beta+\varepsilon\nabla\u\cdot\nabla\r\\
\hskip 4cm =(\na \times \H)\times \H+\mu \Delta \u+(\lambda+\mu)\nabla\Dv\u, \\
\H_t-\nabla\times(\u\times\H)=-\nabla\times(\nu\nabla\times\H),\quad
\Dv\H=0,
\end{cases}
\end{equation}
with the  initial-boundary conditions which will be specified in Section 4,
  where $\beta>0$ is a constant to be determined later, and  $\varepsilon>0, \delta>0$. Taking $\varepsilon\to 0$ and $\delta\to 0$ in \eqref{e7} will give the solution of \eqref{e1} in Theorem \ref{mt}.
We remark that the nonlinear term $(\nabla\times\H)\times\H$ can be dealt with by the idea arising in incompressible Navier-Stokes equations, but there are no estimates good enough to control possible oscillations of the density
$\r$. In order to overcome this difficulty, we adopt the method in Lions \cite{p2} and Feireisl \cite{f2} which is based on the celebrated
weak continuity of the effective viscous flux $p-(\lambda+2\mu)\Dv\u$. More specifically, it can be shown that
$$(a\r_n^\gamma-(\lambda+2\mu)\Dv\u_n)b(\r_n)\rightarrow (a\ov{\r^\gamma}-(\lambda+2\mu)\Dv\u)\ov {b(\r)}$$
weakly in  $L^1(\O\times(0,T))$, where $\r_n$ and $\u_n$ are a suitable sequence of approximate
solutions, and the symbol $\ov {F(v)}$ stands for a weak limit of $\{F(v_n)\}_{n=1}^\infty$.

\begin{Remark}
Similarly to \cite{f10}, our approach also works for the general barotropic flow: $$p(\r)=a\r^\gamma+z(\r), \quad\lim_{\r\rightarrow\infty}\f{z(\r)}{\r^\gamma}\in[0,\infty),$$
where $a>0, \gamma>\f{3}{2}$. Moreover, we also can extend our results to the initial-boundary value problem for \eqref{e1} in an exterior domain by using the method of invading domain (cf. Section 7.11 in \cite{i1}) and the following special type of Orlicz spaces $L^p_q(\O)$ (see Appendix A in \cite{p2}):
$$L^p_q(\O)=\left\{f\in L^1_{loc}(\O): \; f|_{\{|f|<\eta\}}\in L^q(\O), \textrm{ and } f|_{\{|f|\geq\eta\}}\in L^p(\O), \textrm{ for some } \eta>0\right\}.$$
\end{Remark}

\bigskip

\section{The Solvability of The Magnetic Field}
In order to prove the existence of solutions to \eqref{e7} by Faedo-Galerkin method, we need to show that
the following system can be uniquely solved in terms of $\u$:
\begin{equation} \label{g3}
\begin{cases}
\H_t-\nabla\times(\u\times\H)=-\nabla\times(\nu\nabla\times\H),\\
\Dv\H=0,\\
\H(x,0)=\H_0, \quad \H|_{\partial\O}=0.
\end{cases}
\end{equation}

In fact, we have the following properties:

\begin{Lemma}\label{a1}
Let $\u\in C([0,T]; C_0^2(\ov{\O};\R^3))$. Then there exists at most one function
$$\H\in L^2([0,T]; H^1_0(\O))\cap L^\infty([0,T]; L^2(\O))$$
which solves \eqref{g3} in the weak sense on  $\O\times(0,T)$, and satisfies boundary and initial conditions in the sense of traces.
\end{Lemma}
\begin{proof}
Let $\H_1$, $\H_2$ be two solutions of \eqref{g3} with the same data. Then we have
\begin{equation}\label{a2}
(\H_1-\H_2)_t-\nabla\times(\u\times(\H_1-\H_2))=-\nabla\times(\nu\nabla\times(\H_1-\H_2)).
\end{equation}
Multiplying \eqref{a2} by $\H_1-\H_2$, integrating over $\O$, and using the Cauchy-Schwarz inequality, we obtain
\begin{equation*}
\begin{split}
&\f{1}{2}\f{d}{dt}\int_{\O}|\H_1-\H_2|^2\mathrm{d}x
+\nu\int_{\O}|\nabla\times(\H_1-\H_2)|^2 \mathrm{d}x\\
&=\int_{\O}(\u\times(\H_1-\H_2))\cdot(\nabla\times(\H_1-\H_2))\mathrm{d}x\\
&\leq\f{\nu}{2}\int_{\O}|\nabla\times(\H_1-\H_2)|^2\mathrm{d}x
  +\f{1}{2\nu}\int_{\O}|\u\times(\H_1-\H_2)|^2\mathrm{d}x\\
&\leq\f{\nu}{2}\int_{\O}|\nabla\times(\H_1-\H_2)|^2 \mathrm{d}x
 +C(\nu,\parallel \u\parallel_\infty)\int_{\O}|\H_1-\H_2|^2\mathrm{d}x.
\end{split}
\end{equation*}
This implies
\begin{equation}\label{a3}
\begin{split}
&\f{1}{2}\f{d}{dt}\int_{\O}|\H_1-\H_2|^2\mathrm{d}x
+\f{\nu}{2}\int_{\O}|\nabla\times(\H_1-\H_2)|^2\mathrm{d}x\\
&\leq C(\nu,\parallel \u\parallel_\infty)\int_{\O}|\H_1-\H_2|^2\mathrm{d}x.
\end{split}
\end{equation}
Then,  Lemma \ref{a1} follows directly from Gronwall's inequality and \eqref{a3}.
\end{proof}

\begin{Lemma}\label{a4}
Let $\O\subset\R^3$ be a bounded domain of class $C^{2+\kappa}$, $\kappa>0$. Assume that $\u\in C([0,T]; C_0^2(\ov{\O};\R^3))$
is a given velocity field. Then the solution operator $$\u\mapsto \H[\u]$$ assigns to $\u\in C([0,T]; C_0^2(\ov{\O};\R^3))$
the unique solution $\H$ of \eqref{g3}. Moreover, the solution operator $\u\mapsto \H[\u]$ maps bounded sets in $C([0,T]; C_0^2(\ov{\O};\R^3))$ into
bounded subsets of  $Y:=L^2([0,T]; H^1_0(\O))\cap L^\infty([0,T]; L^2(\O))$, and the mapping
$$\u\in C([0,T]; C_0^2(\ov{\O};\R^3))\mapsto \H \in Y$$ is continuous on any bounded subsets of $C([0,T]; C_0^2(\ov{\O};\R^3))$.
\end{Lemma}
\begin{proof}
The uniqueness of the solution to \eqref{g3} is a consequence of Lemma \ref{a1}.
Noticing that
$$\nabla\times(\nabla\times\H)=\nabla(\text{div} \H)-\Delta H,$$
then \eqref{g3} becomes
\begin{equation} \label{g3b}
\begin{cases}
\H_t-\nabla\times(\u\times\H)=\nu \Delta \H,\\
\Dv\H=0,\\
\H(x,0)=\H_0, \quad \H|_{\partial\O}=0,
\end{cases}
\end{equation}
which is a linear parabolic-type problem in $\H$, so the existence of solution can be
obtained by the standard Faedo-Galerkin methods. And from \eqref{a3}, we can conclude that the solution operator $\u\mapsto \H[\u]$
maps bounded sets in $C([0,T]; C_0^2(\ov{\O};\R^3))$ into
bounded subsets of the set $Y=L^2([0,T]; H^1_0(\O))\cap L^\infty([0,T]; L^2(\O))$.

Our next step is to show the solution operator is continuous from any bounded subset of $C([0,T]; C_0^2(\ov{\O};\R^3)$ to
$L^2([0,T]; H^1_0(\O))\cap L^\infty([0,T]; L^2(\O))$.
To this end, let $\{\u_n\}_{n=1}^\infty$ be a bounded sequence in $C([0,T]; C_0^2(\ov{\O};\R^3)$, i.e., $\{\u_n\}_{n=1}^\infty\subset B(0,K)\subset C([0,T]; C_0^2(\ov{\O};\R^3)$ for some $K>0$, and
$$\u_n\rightarrow\u \textrm{ in }C([0,T]; C_0^2(\ov{\O};\R^3)), \text{ as } n\to\infty.$$
Then, we have, denoting $\H[\u]$ by $\H_{\u}$, and $\H[\u_n]$ by $\H_n$,
\begin{equation*}
\begin{split}
&\f{1}{2}\f{d}{dt}\int_{\O}|\H_n-\H_{\u}|^2\mathrm{d}x
+\nu\int_{\O}|\nabla\times(\H_n-\H_{\u})|^2\mathrm{d}x\\
&=\int_{\O}(\u_n\times\H_n-\u\times\H_{\u}))\cdot(\nabla\times(\H_n-\H_{\u}))\mathrm{d}x\\
&\leq\int_{\O}((\u_n-\u)\times\H_n+\u\times(\H_n-\H_{\u}))\cdot(\nabla\times(\H_n-\H_{\u}))\mathrm{d}x\\
&\leq\parallel\u_n-\u\|_\infty\|\H_n\|^2_Y+K\|\H_n-\H_{\u}\|_{L^2(\O)}
\|\H_n-\H_{\u}\|_{H^1_0(\O)}\\
&\leq C^2\|\u_n-\u\|_\infty+c\|\H_n-\H_{\u}\|^2_{L^2(\O)}
+\f{\nu}{2}\|\H_n-\H_{\u}\|^2_{H^1_0(\O)},
\end{split}
\end{equation*}
where, we used the fact that  $\H_n[\u_n]$ is bounded in Y, says, by
$C=C(K,\nu)$. This implies that
\begin{equation}\label{a5}
\begin{split}
&\f{1}{2}\f{d}{dt}\int_{\O}|\H_n-\H_{\u}|^2\mathrm{d}x
+\f{\nu}{2}\int_{\O}|\nabla\times(\H_n-\H_{\u})|^2\mathrm{d}x\\
&\leq C^2\|\u_n-\u\|_\infty+c\|\H_n-\H_{\u}\|^2_{L^2(\O)}
\end{split}
\end{equation}
Integrating \eqref{a5} over $t\in(0,T)$, and then taking the upper limit over $n$ on the both sides, we get, noting that
$\u_n\rightarrow\u \textrm{ in }C([0,T]; C_0^2(\ov{\O};R^3))$,
\begin{equation}\label{a6}
\begin{split}
&\f{1}{2}\limsup_{n}\int_{\O}|\H_n-\H_{\u}|^2\mathrm{d}x+\f{\nu}{2}\limsup_{n}\int_0^t\int_{\O}|\nabla\times(\H_n-\H_{\u})|^2
\mathrm{d}x\mathrm{d}s\\
&\leq c\limsup_{n}\int_0^t\|\H_n-\H_{\u}\|^2_{L^2(\O)}\mathrm{d}s
\leq c\int_0^t\limsup_{n}\|\H_n-\H_{\u}\|^2_{L^2(\O)}\mathrm{d}s,
\end{split}
\end{equation}
thus, from \eqref{a6}, using Gronwall's inequality and the same initial value for
$\H_n$ and $\H_{\u}$, we get
$$\limsup_{n}\int_{\O}|\H_n-\H_{\u}|^2(t)\mathrm{d}x=0.$$
This yields, from \eqref{a6} again,
$$\limsup_{n}\int_0^t\int_{\O}|\nabla\times(\H_n-\H_{\u})|^2
\mathrm{d}x\mathrm{d}s=0.$$
Therefore, we obtain
$$\H_n\rightarrow\H_{\u} \textrm{ in } Y.$$
This completes the proof of the continuity of the solution operator.
\end{proof}

\bigskip

\section{The Faedo-Galerkin Approximation Scheme}

In this section, we establish the existence of solutions to \eqref{e7} following
the approach  in \cite{f1} with the extra efforts to overcome the difficulty arising from  the magnetic field.
%the necessary modifications to accommodate the extra magnetic equation.
Let $$X_n=\mathrm{span}\{\eta_j\}_{j=1}^n$$ be the finite-dimensional
space endowed with the $L^2$ Hilbert space structure, where the functions
$\eta_j\in \mathcal{D}(\O;\R^3)$, $j=1,2,...$,  form a dense subset
in, says, $C_0^2(\ov\O;\R^3)$. Through this paper, we use $\mathcal{D}$ to denote
$C_0^\infty$, and $\mathcal{D}'$ for the sense of distributions.
The approximate velocity field $\u_n\in C([0,T];X_n)$ satisfies
the system of integral equations:
\begin{equation}\label{g1}
\begin{split}
&\int_{\O}\r\u_n(x,t)\cdot\eta\,\mathrm{d}x-\int_{\O} \m_{0,\dl}\cdot\eta\,\mathrm{d}x\\
&=\int_0^{t}\int_{\O}\Big(\mu\Delta\u_n-\Dv(\r\u_n\otimes\u_n)+\na\left((\lambda+\mu)\Dv\u_n-a\r^\gamma
-\delta\r^\beta\right)\\
&\qquad\qquad\quad-\varepsilon\na\r\cdot\na\u_n+(\na \times \H)\times \H\Big)\cdot\eta\,\mathrm{d}x\mathrm{d}\tau,
\end{split}
\end{equation}
for any $t\in[0,T]$ and any $\eta\in X_n$, where $\varepsilon$,
$\delta$, and $\beta$ are fixed positive parameters.
The density $\r_n=\r[\u_n]$ is determined uniquely as the
solution of the Neumann initial-boundary value problem (cf. Lemma
2.2 in \cite{f1}):
\begin{equation}\label{g2}
\begin{cases}
\rho_t +\Dv(\r\u_n)=\varepsilon\Delta\r, \\
\na\r\cdot {\bf{n}}|_{\partial\O}=0, \\
\r(x,0)=\r_{0,\delta}(x);
\end{cases}
\end{equation}
and the magnetic field $\H_n=\H[\u_n]$  as a solution to the system \eqref{g3}.
The initial data $\r_{0,\delta}$ is a smooth function in
$C^3(\ov{\O})$ satisfying the homogeneous Neumann boundary
conditions $\na\r_{0,\delta}\cdot {\bf{n}}|_{\partial\O}=0$, and
\begin{gather}
0<\delta\leq\r_{0,\delta}\leq\dl^{-\f{1}{2\beta}}, \label{e8} \\
\r_{0,\delta}\rightarrow \r_0 \textrm{ in } L^\gamma(\O),\quad
|\{\r_{0,\delta}<\r_0\}|\rightarrow 0 \textrm{ for } \dl\rightarrow
0; \label{e9}
\end{gather}
moreover, we set
\begin{equation}\label{e10}
\m_{0,\dl}(x)=\begin{cases} \m_0(x),& \textrm{ if }\; \r_{0,\delta}\geq\r_0(x),\\
0,& \textrm{ if } \; \r_{0,\delta}<\r_0(x).
\end{cases}
\end{equation}

Due to Lemma \ref{a1} and Lemma \ref{a4}, the problem \eqref{g1}, \eqref{g2}, and \eqref{g3} can be solved locally
in time by means of the Schauder fixed-point technique; see for example Section 7.2 of Feireisl \cite{f2}. As  in \cite{f1, f2}
the role of the "artificial pressure" term $\delta\r^\beta$ in \eqref{g1} is to provide additional estimates
on the approximate densities in order to
facilitate the limit passage $\varepsilon\rightarrow 0$ (cf. Chapter 7 in \cite{f2}). To this end, one has to take
$\beta$ large enough, says, $\beta>8$, and to re-parametrize the initial distribution of the approximate densities
so that
\begin{equation}\label{g7}
\delta\int_{\O}\r_{0,\dl}^\beta\,\mathrm{d}x\rightarrow 0 \textrm{ as } \dl\rightarrow 0.
\end{equation}

To obtain uniform bounds on $\u_n$, we derive an energy equality similar to \eqref{f5} as follows. Taking $\eta(x)=\u_n(x,t)$ with fixed $t$ in \eqref{g1} and repeating the procedure for  \textit{a priori} estimates in Section 2,
we deduce a ``kinetic energy equality":
\begin{equation}\label{g4}
\begin{split}
&\f{d}{dt}\int_{\O}\left(\f{1}{2}\r_n\u_n^2+\f{a}{\gamma-1}\r_n^\gamma
+\f{\delta}{\beta-1}\r_n^\beta+
\f{1}{2}|\H_n|^2\right)\,\mathrm{d}x\\
&+\int_{\O}\left(\mu|D\u_n|^2+(\lambda+\mu)(\Dv\u_n)^2
+\nu(\nabla\times\H_n)\cdot(\nabla\times\H_n)\right)\,\mathrm{d}x\\
&+\varepsilon\int_{\O}\left(a\gamma\r_n^{\gamma-2}+\delta\beta\r_n^{\beta-2}\right)
|\nabla\r_n|^2\,\mathrm{d}x = 0.
\end{split}
\end{equation}
The uniform estimates obtained from \eqref{g4} furnish the possibility to repeat the above fixed point argument
to extend the local solution $\u_n$ to the whole time interval $[0,T]$. Then, by the solvability of equations
\eqref{g2} and \eqref{g3}, we obtain the functions $\{\r_n, \H_n\}$ on the whole time interval $[0,T]$.

The next step in the proof of Theorem \ref{mt} consists of passing to the limit as $n\rightarrow \infty$ in
the sequence of approximate solutions $\{\r_n, \u_n, \H_n\}$ obtained above.
We first observe that the terms related to $\u_n$ and $\r_n$ can be treated similarly to
Section 7.3.6 in \cite{f2}, due to the energy equality \eqref{g4}. It remains to show the convergence of the sequence of solutions $\{\H_n\}_{n=1}^\infty$.
From \eqref{g4}, we conclude
$$\H_n \textrm{ is bounded in } L^\infty([0,T]; L^2(\O))\cap L^2([0,T]; H_0^1(\O)).$$
%and $$\H_n \textrm{ is bounded in } L^2([0,T]; H_0^1(\O)).$$
This implies that, by the compactness of $H^1(\O)\hookrightarrow L^2(\O)$ and selecting a subsequence if necessary,
there exists a function $\H\in L^\infty([0,T]; L^2(\O))\cap L^2([0,T]; H^1(\O))$ with $\Dv\H=0$ such that
$\H_n(\cdot,t)\rightarrow \H(\cdot,t) \textrm{ in } L^2(\O)$ for a.e. $t\in [0,T]$. Thus,
$$(\na \times \H_n)\times \H_n\rightarrow(\na \times \H)\times \H \textrm{ in } \mathcal{D}'(\O\times(0,T)).$$
Similarly, we have
$$\nabla\times(\u_n\times\H_n)\rightarrow\nabla\times(\u\times\H)\textrm{ in } \mathcal{D}'(\O\times(0,T)),$$
where
$$\u_n\rightarrow \u \textrm{ weakly in } L^2([0,T]; H_0^1(\O)).$$

Therefore, \eqref{g1} and \eqref{g3} holds at least in the sense of distribution. Moreover, by the uniform estimates on
$\u$, $\H$ and the third equation in \eqref{e1}, we know that the map
$$t\rightarrow\int_{\O}\H_n(x,t)\varphi(x)\,\mathrm{d}x \textrm{ for any } \varphi\in \mathcal{D}(\O),$$
is equi-continuous on [0,T].
By the Ascoli-Arzela Theorem, we know that
$$t\rightarrow\int_{\O}\H(x,t)\varphi(x)\,\mathrm{d}x$$
is continuous for any $\varphi\in \mathcal{D}(\O)$. Thus, $\H$ satisfy the initial condition in \eqref{g3} in the sense of distribution.

Now, we are ready to summarize an existence result for problem \eqref{g1}, \eqref{g2}, and \eqref{g3} (cf. Proposition 7.5 in \cite{f2}).

\begin{Proposition}\label{g5}
Assume that $\Omega \subset \R^3$ is a bounded domain of the class
$C^{2+\kappa}$, $\kappa >0$. Let $\varepsilon>0$, $\delta>0$,
and $\beta>max\{4, \gamma\}$ be fixed. Then for any given $T>0$,
problem \eqref{g1}, \eqref{g2}, and \eqref{g3} admits at least one solution $\r, \u,
\H$ in the following sense:
\begin{itemize}
\item[(1)] The density $\r$ is a non-negative function such that
$$\r\in L^r([0,T]; W^{2,r}(\O)), \quad\partial_t\r\in L^r((0,T)\times\O),$$
for some $r>1$,
the velocity $\u$ belongs to the class $L^2([0,T]; H_0^1(\O))$,
equation \eqref{g2} holds a.e on $\O\times(0,T)$, and the boundary
condition as well as the initial data condition on $\r$ are
satisfied in the sense of traces. Moreover, the total mass is
conserved, especially,
$$\int_{\O}\r(x,t)\,\mathrm{d}x=\int_{\O}\r_{0,\delta}(x)\,\mathrm{d}x,$$
for all $ t\in [0,T]$;
and the following estimates hold:
$$\delta\int_0^T\!\!\!\!\int_{\O}\r^{\beta+1}\,\mathrm{d}x\mathrm{d}t\leq C(\varepsilon),$$
$$\varepsilon\int_0^T\!\!\!\!\int_{\O}|\nabla\r|^2\,\mathrm{d}x\mathrm{d}t\leq C \textrm{ with C independent of } \varepsilon.$$
\item[(2)] All quantities appearing in equation \eqref{g1} are locally integrable, and the equation is satisfied in
$\mathcal{D}'(\O\times (0,T))$. Moreover, we have
$$\r\u\in C([0,T]; L^{\f{2\gamma}{\gamma+1}}_{weak}(\O)),$$
and $\r\u$ satisfies the initial condition.
\item[(3)] All terms in \eqref{g3} are locally integrable on $\O\times (0,T)$. The magnetic field $\H$ satisfies equation
\eqref{g3} and the initial data in the sense of distribution. $\Dv\H=0$ also holds in the sense of distribution.
\item[(4)] The energy inequality
\begin{equation}\label{g6}
\begin{split}
&\int_{\O}\left(\f{1}{2}\r\u^2+\f{a}{\gamma-1}\r^\gamma
  +\f{\delta}{\beta-1}\r^\beta+\f{1}{2}|\H|^2\right)(x,t)\,\mathrm{d}x\\
&+\int_{\O}\left(\mu|D\u|^2+\lambda(\Dv\u)^2+\nu|\nabla\times\H|^2\right)(x,t)\,\mathrm{d}x\\
&+\varepsilon\int_{\O}\left(a\gamma\r^{\gamma-2}+\delta\beta\r^{\beta-2}\right)
  |\nabla\r|^2(x,t)\,\mathrm{d}x\\
&\leq\int_{\O}\left(\f{1}{2}\f{|\m_0|^2}{\r_0}
+\f{a}{\gamma-1}\r_0^\gamma+\f{\delta}{\beta-1}\r_0^\beta+\f{1}{2}|\H_0|^2\right)
\,\mathrm{d}x,
\end{split}
\end{equation}
holds  a.e $t\in[0,T]$.
\end{itemize}
\end{Proposition}

In the next section, we will complete the proof of Theorem \ref{mt} by taking vanishing artificial viscosity and vanishing artificial pressure.
%Noticing that the techniques used in vanishing the effect
%of the viscosity term and in passing to the limit in the artificial pressure term are %rather similar as explained in Chap.7 in \cite{f2},
%(See also \cite{f1}). Due to this reason, we will focus in this paper only in the step of %passing to the limit in the artificial pressure term.

\bigskip

\section{The Convergence of the Approximate Solution Sequence}

Now we have the approximate solutions $\{\r_{\varepsilon, \delta}, \u_{\varepsilon, \delta}, \H_{\varepsilon, \delta}\}$ obtained in Section 4. To prove  Theorem \ref{mt} we need to
take the limits as the artificial viscosity coefficient $\varepsilon\to 0$ and as the artificial pressure coefficient $\delta\to 0$.

First, following Chapter 7 in \cite{f2} (see also \cite{f1}), we can pass to the limit as $\varepsilon\to 0$ to obtain the following result:

\begin{Proposition}\label{m1}
\textit{Assume $\Omega \subset \R^3$ is a bounded domain of  class
$C^{2+\kappa}$, $\kappa >0$. Let $\delta>0$,
and $$\beta>max\{4, \f{6\gamma}{2\gamma-3}\}$$ be fixed. Then, for given initial data $\r_0, \m_0$ as in \eqref{e8}-\eqref{e10} and $\H_0$ as in \eqref{e4}, there
exists a finite energy weak solution $\r, \u, \H$ of the problem:
\begin{equation} \label{m2}
\begin{cases}
\rho_t +\Dv(\r\u)=0, \\
(\r\u)_t+\Dv\left(\r\u\otimes\u\right)+\nabla (a\r^\gamma+\delta\r^\beta)=(\na \times \H)\times \H+\mu \Delta \u+(\lambda+\mu)\u, \\
\H_t-\nabla\times(\u\times\H)=-\nabla\times(\nu\nabla\times\H),\quad
\Dv\H=0,
\end{cases}
\end{equation}
satisfying the initial boundary conditions \eqref{e8}-\eqref{g7} and \eqref{e4}.
Moreover, $\r\in L^{\beta+1}(\O\times (0,T))$ and the continuity equation in \eqref{m2} holds in the sense of renormalized solutions.
Furthermore, $\r, \u, \H$ satisfy the following uniform estimates:
\begin{equation}\label{ml3}
\sup_{t\in[0,T]}\|\r(t)\|^\gamma_{L^\gamma(\O)}\leq cE_\delta[\r_0, \m_0, \H_0],
\end{equation}
\begin{equation}\label{m4}
\dl\sup_{t\in[0,T]}\|\r(t)\|^\beta_{L^\beta(\O)}\leq cE_\delta[\r_0, \m_0, \H_0],
\end{equation}
\begin{equation}\label{m5}
\sup_{t\in[0,T]}\|\sqrt{\r(t)}\u(t)\|^2_{L^2(\O)}\leq cE_\delta[\r_0, \m_0, \H_0],
\end{equation}
\begin{equation}\label{m6}
\|\u\|_{L^2([0,T];H^1_0(\O))}\leq cE_\delta[\r_0, \m_0, \H_0],
\end{equation}
\begin{equation}\label{m7}
\sup_{t\in[0,T]}\|\H(t)\|^2_{L^2(\O)}\leq cE_\delta[\r_0, \m_0, \H_0],
\end{equation}
\begin{equation}\label{m8}
\|\H\|_{L^2([0,T];H^1_0(\O))}\leq cE_\delta[\r_0, \m_0, \H_0],
\end{equation}
where the constant c is independent of $\delta>0$ and
$$E_\delta[\r_0, \m_0, \H_0]=\int_{\O}\left(\f{1}{2}\f{|\m_{0,\dl}|^2}{\r_{0,\dl}}+\f{a}{\gamma-1}\r_{0,\dl}^\gamma+
\f{\delta}{\beta-1}\r_{0,\dl}^\beta+\f{1}{2}|\H_0|^2\right)\,\mathrm{d}x.$$}
\end{Proposition}

Observing that the conditions \eqref{e9}-\eqref{g7} imply that the term $E_\delta[\r_0, \m_0, \H_0]$ appearing
in \eqref{ml3}-\eqref{m8} can be indeed majored to
be a constant $E[\r_0, \m_0, \H_0]$ which is also independent of the choice of $\delta$. This gives us a list of
\textit{a priori} estimates on $\r, \u, \H$ which are independent of $\dl$.
We omit the proof of Proposition \ref{m1}, and
concentrate our attention on passing to the limit in the artificial
pressure term  to establish the weak sequential stability property for the approximate solutions obtained in  Proposition \ref{m1} as $\delta\rightarrow 0$.

\subsection{On the integrability of the density}
We first derive an estimate of the density $\r_\delta$
uniform in  $\dl>0$ to make possible passing to the limit
in the term $\dl\r_\dl^\beta$ as $\dl
\rightarrow 0$. The technique is similar to that in \cite{f1}.

Noting that the function $b(\r)=\textrm{ln}(1+\r)$ satisfies the
condition \eqref{e12}, and $\r_\dl, \u_\dl, \H_\dl$ are the solution to \eqref{m2} in
the sense of renormalized solutions, we have
\begin{equation}\label{x1}
\left(\ln(1+\r_\dl)\right)_t+\Dv\left(\textrm{ln}\left(1+\r_\dl\right)
\u_\dl\right)+\left(\f{\r_\dl}{1+\r_\dl}-\textrm{ln}\left(1+\r_\dl\right)\right)\Dv
\u_\dl=0.
\end{equation}

Now we introduce an auxiliary operator
$${B}: \; \left\{f\in L^p(\O):\; \int_{\O}f=0\right\}\mapsto
[W^{1,p}_0(\O)]^3$$ which is a bounded linear operator, i.e.,
\begin{equation}\label{m9}
\| {B}[f]\|_{W^{1,p}_0(\O)}\leq c(p)\|
f\|_{L^p(\O)} \textrm{ for any } 1<p<\infty;
\end{equation}
and the function $W=B[f]\in\R^3$ solves the problem
\begin{equation}\label{m10}
\Dv \mathrm{W}=f\textrm{ in }\O, \quad \mathrm{W}|_{\partial\O}=0.
\end{equation}
Moreover, if f can be written in the form $f=\Dv g$ for some $g\in L^r$, $g\cdot {\bf{n}}|_{\partial\O}=0$,
then
\begin{equation}\label{m16}
\|{B}[f]\|_{L^r(\O)}\leq c(r)\| g\|_{L^r(\O)}
\end{equation}
for arbitrary $1<r<\infty$.

Define the functions:
$$\varphi_i=\psi(t)B_i\left[\textrm{ln}(1+\r_\dl)-\dashint_{\O}\textrm{ln}(1+\r_\dl)\,\mathrm{d}x\right],
\; \psi\in \mathcal{D}(0,T), \; i=1,2,3,$$
where $\dashint_{\O}\textrm{ln}(1+\r_\dl)\,\mathrm{d}x=\frac1{|\O|}\int_\O\ln(1+\r_\dl) dx$ is the average of $ \ln(1+\r_\dl)$ over $\O$.
By virtue of \eqref{ml3} and \eqref{x1}, we get
$$
\textrm{ln}(1+\r_\dl)\in C([0,T]; L^p(\O)) \textrm{
for any finite } p>1.
$$
Therefore, from \eqref{m9}, we have
$$
\varphi_i \in C([0,T]; W^{1,p}_0(\O)) \textrm{ for any finite }
p>1.
$$
In particular, $\varphi_i \in C(\O\times[0,T])$ by the  Sobolev embedding theorem.
Consequently,  $\varphi_i$ can be used as test
functions for the momentum balance equation in \eqref{m2}. After a little bit lengthy but
straightforward computation, we obtain:
\begin{equation}\label{m13}
\int_0^T\!\!\!\!\int_{\O}\psi(\dl
\r_\dl^\beta+a\r_\dl^\gamma)\textrm{ln}(1+\r_\dl)\,\mathrm{d}x\mathrm{d}t=\sum_{j=1}^7I_j,
\end{equation}
where
\begin{equation*}
\begin{split}
I_1=&\int_0^T\psi\int_{\O}(\dl
\r_\dl^\beta+a\r_\dl^\gamma)\,\mathrm{d}x\dashint_{\O}\textrm{ln}(1+\r_\dl)\,\mathrm{d}x\mathrm{d}t,\\
I_2=&(\lambda+\mu)\int_0^T\!\!\!\!\int_{\O}\psi\textrm{ln}(1+\r_\dl)\Dv\u_\dl\,\mathrm{d}x\mathrm{d}t,\\
I_3=&-\int_0^T\!\!\!\!\int_{\O}\psi_t\r_\dl u^i_\dl
B_i\left[\textrm{ln}(1+\r_\dl)-\dashint_{\O}\textrm{ln}(1+\r_\dl)\,\mathrm{d}x\right]\,\mathrm{d}x\mathrm{d}t,\\
I_4=& \int_0^T\!\!\!\!\int_{\O}\psi(\mu\partial_{x_j}u_\dl^i-\r_\dl u_\dl^i
u_\dl^j)\partial_{x_j}B_i\left[\textrm{ln}(1+\r_\dl)-\dashint_{\O}\textrm{ln}(1+\r_\dl)\,\mathrm{d}x\right]
\,\mathrm{d}x\mathrm{d}t,\\
I_5=&\int_0^T\!\!\!\!\int_{\O}\psi \r_\dl u_\dl^i
B_i\left[\left(\textrm{ln}(1+\r_\dl)-\f{\r_\dl}{1+\r_\dl}\right)\Dv\u_\dl \right. \\
&\qquad\qquad\qquad\qquad \left. -\dashint_{\O}\left(\textrm{ln}(1+\r_\dl)-\f{\r_\dl}{1+\r_\dl}\right)\Dv\u
_\dl\,\mathrm{d}x\right] \,\mathrm{d}x\mathrm{d}t,\\
I_6=&\int_0^T\!\!\!\!\int_{\O}\psi \r_\dl u_\dl^iB_i[\Dv(\textrm{ln}(1+\r_\dl)\u_\dl]\,\mathrm{d}x\mathrm{d}t,\\
I_7=&\int_0^T\!\!\!\!\int_{\O}\psi(\nabla\times\H_\dl)\times\H_\dl\cdot B_i\left[\dashint_{\O}\textrm{ln}
(1+\r_\dl-\textrm{ln}(1+\r_\dl))\,\mathrm{d}x\right]\,\mathrm{d}x\mathrm{d}t.
\end{split}
\end{equation*}
Now, we can estimate the integrals $I_1-I_7$ as follows.

\noindent (1) First, we see that
%\begin{equation*}
%I_1=\int_0^T\psi\int_{\O}(\dl
%\r_\dl^\beta+a\r_\dl^\gamma)\,\mathrm{d}x
%\dashint_{\O}\textrm{ln}(1+\r_\dl)\,\mathrm{d}x\mathrm{d}t
%\end{equation*}
$I_1$  is bounded uniformly in  $\dl$,
from \eqref{ml3}, \eqref{m4}, and the following property:
\begin{equation*}
\lim_{t \rightarrow \infty}
\frac{\textrm{ln}(1+t)}{t^\gamma}=0.
\end{equation*}

\noindent (2) As for the second term, we also have
\begin{equation*}
|I_2|\le \int_0^T\!\!\!\!\int_{\O}\left|\psi
\textrm{ln}(1+\r_\dl)\Dv
\u_\dl\right| \,\mathrm{d}x\mathrm{d}t\leq c,
\end{equation*}
by the H\"{o}lder inequality, \eqref{m6}, \eqref{ml3}, and the following property:
\begin{equation*}
\lim_{t \rightarrow \infty} \frac{\textrm{ln}^2(1+t)}{t^\gamma}=0,
\end{equation*}
where and throughout the rest of the paper, $c>0$ denotes a generic constant.

\noindent (3) Similarly, for the third term, we have
\begin{equation*}
|I_3|\le \int_0^T\!\!\!\!\int_{\O}\left|\psi_t\r_\dl u^i_\dl
B_i\left[\textrm{ln}(1+\r_\dl)-\dashint_{\O}\textrm{ln}(1+\r_\dl)\,\mathrm{d}x\right]
\right|\,\mathrm{d}x\mathrm{d}t\leq c. \end{equation*}
Here, we have used \eqref{m5}, \eqref{m6}, and the
embedding $W^{1,p}(\O)\hookrightarrow L^{\infty}(\O)$ for
$p>3$, since
$\textrm{ln}(1+\r_\dl)-\dashint_{\O}\textrm{ln}(1+\r_\dl)\,\mathrm{d}x\in
L^p(\O)$ for any $1<p<\infty$.

\noindent (4) Similarly to (3), we have
\begin{equation*}
\left|\int_0^T\!\!\!\!\int_{\O}\psi\partial_{x_j}u_\dl^i\partial_{x_j}
B_i\left[\textrm{ln}(1+\r_\dl)-\dashint_{\O}\textrm{ln}(1+\r_\dl)\,\mathrm{d}x\right]
\,\mathrm{d}x\mathrm{d}t\right|\leq c,
\end{equation*}
and, by \eqref{ml3}, \eqref{m6}, and H\"{o}lder inequality, we have
\begin{equation*}
\left|\int_0^T\!\!\!\!\int_{\O}\psi\r_\dl u_\dl^i
u_\dl^j\partial_{x_j}
B_i\left[\textrm{ln}(1+\r_\dl)-\dashint_{\O}\textrm{ln}(1+\r_\dl)\,\mathrm{d}x\right]
\,\mathrm{d}x\mathrm{d}t\right|\leq c.
\end{equation*}
Here, we used the restriction $\gamma>\f{3}{2}$.
Therefore, we obtain
\begin{equation*}
|I_4|
%=|\int_0^T\!\!\!\!\int_{\O}\psi(\mu\partial_{x_j}u_\dl^i-\rho_\dl u_\dl^i %u_\dl^j\partial_{x_j})\partial_{x_j}
%B_i[\textrm{ln}(1+\r_\dl)-\dashint_{\O}\textrm{ln}(1+\r_\dl)\,\mathrm{d}x]
%\,\mathrm{d}x\mathrm{d}t|
\leq c.
\end{equation*}

\noindent (5) Next, by H\"{o}lder inequality and \eqref{m9}, we have,
\begin{equation*}
\begin{split}
|I_5|%=&|\int_0^T\!\!\!\!\int_{\O}\psi \r_\dl u_\dl^i
%B_i[\left(\textrm{ln}(1+\r_\dl)-\f{\r_\dl}{1+\r_\dl}\right)\Dv\u_\dl\\&-
%\dashint_{\O}
%\left(\textrm{ln}(1+\r_\dl)-\f{\r_\dl}{1+\r_\dl}\right)\Dv\u_\dl\,\mathrm{d}x]
%\,\mathrm{d}x\mathrm{d}t|\\
\leq c\int_0^T|\psi|
\|\r_\dl\|^{\f{1}{2}}_{L^\gamma(\O)}\|\sqrt{\r_\dl} u_\dl\|_{L^2(\O)}
\| B_i [w]\|_{L^{\f{2\gamma}{\gamma-1}}(\O)}\mathrm{d}t\leq c,
\end{split}
\end{equation*}
since $$w:=\left(\textrm{ln}(1+\r_\dl)-\f{\r_\dl}{1+\r_\dl}\right)\Dv\u_\dl\\-
\dashint_{\O}
\left(\textrm{ln}(1+\r_\dl)-\f{\r_\dl}{1+\r_\dl}\right)\Dv\u_\dl\,\mathrm{d}x\in L^r(\O),$$
for some $1<r<2,$ and here we have used the estimates \eqref{ml3}, \eqref{m5}.

\noindent (6) Similarly to (5), using \eqref{ml3}, \eqref{m5}, and we have
\begin{equation*}
|I_6|\leq\int_0^T\!\!\!\!\int_{\O}\left|\psi \r_\dl u_\dl^i
B_i[\Dv(\textrm{ln}(1+\r_\dl)\u_\dl)]\right|\,\mathrm{d}x\mathrm{d}t \leq c.
\end{equation*}
Here, we have also used the property \eqref{m16}.

\noindent (7) Finally, using H\"{o}lder inequality again, we have
\begin{equation*}
\begin{split}
|I_7|%\leq &|\int_0^T\!\!\!\!\int_{\O}\psi(\nabla\times\H_\dl)\times\H_\dl\cdot %B\left[\dashint_{\O}\textrm{ln}
%(1+\r_\dl-\textrm{ln}(1+\r_\dl))\,\mathrm{d}x\right]\,\mathrm{d}x\mathrm{d}t|\\
\leq c\int_0^T|\psi|\|\nabla\times\H_\dl\|_{L^2(\O)}\|\H_\dl\|_{L^2(\O)}\mathrm{d}t\leq c.
\end{split}
\end{equation*}
Here we used the result $\varphi_i \in C([0,T]\times\O)$, \eqref{m7}, and \eqref{m8}.

Consequently, we have proved the following result:

\begin{Lemma}\label{m15}
The solutions $\r_\dl$ of system \eqref{m1} also
satisfies the following estimate
$$\int_0^T\!\!\!\!\int_{\O}\psi(\dl
\r_\dl^\beta+a\r_\dl^\gamma)\mathrm{ln}(1+\r_\dl)\,\mathrm{d}x\mathrm{d}t\leq
c,$$ where the constant c is independent of $\dl >0.$
\end{Lemma}

\begin{Remark}\label{bb1}
Lemma \ref{m15} yields
$$\int_0^T\!\!\!\!\int_{\O}(\dl
\r_\dl^\beta+a\r_\dl^\gamma)\textrm{ln}(1+\r_\dl)\,\mathrm{d}x\mathrm{d}t\leq
c,$$ where c does not depend on $\dl$.
Using the similar method to Lemma 4.1 in \cite{f1}, it can be shown (cf.
\cite{p2, f1, f2}) that the optimal estimate for the density
$\r_\dl$ is the following:
\begin{equation*}
\int_0^T\!\!\!\!\int_{\O}(\dl
\r_\dl^\beta+a\r_\dl^\gamma)\r^\theta_\dl\,\mathrm{d}x\mathrm{d}t\leq
c,
\end{equation*}
where the constant c is independent of $\dl >0$, and
 $\theta>0$ is a constant. But as shown later, our estimate in Lemma \ref{m15}
is enough for our purpose.
\end{Remark}

Define the set
\begin{equation*}
J^{\dl}_k=\{(x,t)\in (0,T)\times \Omega: \; \r_\dl
(x,t)\leq k\}, \qquad  k>0, \quad \dl \in (0,1).
\end{equation*}
From \eqref{ml3}, there exists a constant $s\in(0,\infty)$
such that, for all $\dl \in (0,1)$ and $k>0$,
$$
\textrm{meas}\{\O\times(0,T)-J_k^\dl\}\leq \f{s}{k}.
$$
We have the following  estimate:
\begin{equation}\label{m17}
\begin{split}
\left|\int_0^T\!\!\!\!\int_{\O}\dl
\r_\dl^\beta\,\mathrm{d}x\mathrm{d}t\right|
&\leq \iint_{J_k^\dl}\dl\r_\dl^\beta\,\mathrm{d}x\mathrm{d}t+
\iint_{\O\times (0,T)-J_k^\dl}\dl\r_\dl^\beta\,\mathrm{d}x\mathrm{d}t\\
&\leq T\dl k^\beta \textrm{meas}\{\O\}+\dl
\int_0^T\!\!\!\!\int_{\O}\chi_{\O\times (0,T)-J_k^\dl}\r_\dl^\beta\,\mathrm{d}x\mathrm{d}t.
\end{split}
\end{equation}
Then, by the H\"{o}lder inequality in Orlicz spaces (cf. \cite{a1}) and  Lemma \ref{m15}, we obtain
\begin{equation}\label{m18}
\begin{split}
&\dl\int_0^T\!\!\!\!\int_{\O}\chi_{\O\times (0,T)-J_k^\dl}\r_\dl^\beta\,\mathrm{d}x\mathrm{d}t
\leq\dl\|\chi_{\O\times (0,T)-J_k^\dl}\|_{L_N}\max\{1,\int_0^T\!\!\!\!\int_{\O}
M(\r_\dl^\beta)\,\mathrm{d}x\mathrm{d}t\}\\
&\leq\dl\left(N^{-1}\left(\f{k}{s}\right)\right)^{-1}\max\{1, \int_0^T\!\!\!\!\int_{\O}2
(1+\r_\dl^\beta)\textrm{ln}(1+\r_\dl^\beta)\,\mathrm{d}x\mathrm{d}t\}\\
&\leq \dl\left(N^{-1}\left(\f{k}{s}\right)\right)^{-1}\max\{1, (4\textrm{ln}2)T\,\textrm{meas}\{\O\}+4\beta\int_0^T\!\!\!\!\int_{\O\cap\{\r_\dl\geq 1\}}
\r_\dl^\beta\textrm{ln}(1+\r_\dl)\,\mathrm{d}x\mathrm{d}t\}\\
&\leq \left(N^{-1}\left(\f{k}{s}\right)\right)^{-1}\max\{\dl, (4\textrm{ln}2)\dl
T\,\textrm{meas}\{\O\}+4\dl\beta\int_0^T\!\!\!\!\int_{\O}
\r_\dl^\beta\textrm{ln}(1+\r_\dl)\,\mathrm{d}x\mathrm{d}t\},
\end{split}
\end{equation}
where $L_M(\Omega)$, and $L_N(\Omega)$ are two Orlicz Spaces generated by two complementary N-functions
\begin{equation}\label{l1}
\begin{split}
&M(s)=(1+s)\textrm{ln}(1+s)-s,\\ &N(s)=e^s-s-1,
\end{split}
\end{equation}
respectively.
Due to Lemma \ref{m15}, we know that, if $\delta<1$,
$$\max\{\dl, (4\textrm{ln}2)\dl T\,\textrm{meas}\{\O\}+4\dl\beta\int_0^T\!\!\!\!\int_{\O}
\r_\dl^\beta\textrm{ln}(1+\r_\dl)\,\mathrm{d}x\mathrm{d}t\}\leq c,$$
for some $c>0$ which is independent of $\delta$.
Combining \eqref{m17} with \eqref{m18}, we obtain the estimate
\begin{equation*}
\left|\int_0^T\!\!\!\!\int_{\O}\dl
\r_\dl^\beta\,\mathrm{d}x\mathrm{d}t\right|\leq T\dl
k^\beta
\textrm{meas}\{\O\}+c\left(N^{-1}\left(\f{k}{s}\right)\right)^{-1},
\end{equation*}
where c does not depend on $\dl$ and k. Consequently
\begin{equation}\label{m19}
\limsup_{\dl\rightarrow 0}\left|\int_0^T\!\!\!\!\int_{\O}\dl
\r_\dl^\beta\,\mathrm{d}x\mathrm{d}t\right|\leq
c\left(N^{-1}\left(\f{k}{s}\right)\right)^{-1}.
\end{equation}
The right-hand side of \eqref{m19} tends to zero as $k\rightarrow\infty$.
Thus, we have
\begin{equation*}
\lim_{\dl\rightarrow 0}\int_0^T\!\!\!\!\int_{\O}\dl
\r_\dl^\beta\,\mathrm{d}x\mathrm{d}t=0,
\end{equation*}
which yields
\begin{equation}
\dl \r_\dl^\beta \rightarrow 0  \textrm{ in }
\mathcal{D}'(\O\times (0,T)).
\end{equation}

\subsection{Passing to the limit}
The uniform estimates on $\r$ in Lemma \ref{m15}, and Proposition \ref{m1} imply,
as $\dl\to 0$,
\begin{equation}\label{m20}
\r_\dl\rightarrow \r  \textrm{   in   }
C([0,T];L^\gamma_{weak}(\O)),
\end{equation}
\begin{equation}\label{m21}
\u_\dl\rightarrow \u
\textrm{   weakly  in   } L^2([0,T];H_0^1(\O)),
\end{equation}
and
\begin{equation}\label{m22}
\begin{split}
&\H_\dl\rightarrow \H
\textrm{   weakly*  in   } L^2([0,T];H_0^1(\O))\cap L^\infty([0,T];L^2(\O)), \\
&\Dv\H=0 \textrm{ in } \mathcal{D}'(\O\times (0,T));
\end{split}
\end{equation}
and, from Lemma \ref{m15} and Proposition 2.1 in \cite{f2}, we have, as $\dl\to 0$,
\begin{equation}\label{m23}
\rho_\dl^\gamma\rightarrow \ov{\r^\gamma} \textrm{
weakly in } L^1([0,T];L^1(\O)),
\end{equation}
subject to a subsequence.

By \eqref{m21}, \eqref{m22} and the compactness of $H_0^1(\O)\hookrightarrow L^2(\O)$, we obtain,
\begin{equation}\label{m24}
\nabla\times(\u_\delta\times\H_\delta)\rightarrow\nabla\times(\u\times\H) \quad\textrm{ in } \mathcal{D}'(\O\times (0,T)),
\end{equation}
and
\begin{equation}\label{m25}
(\na \times \H_\dl)\times \H_\dl\rightarrow(\na \times \H)\times \H \quad\textrm{ in } \mathcal{D}'(\O\times (0,T)),
\end{equation}
as $\dl\to 0$.
On the other hand, by virtue of the momentum balance in \eqref{m2} and estimates \eqref{ml3}-\eqref{m8},
we have, as $\dl\to 0$,
\begin{equation}\label{x4}
\r_\dl\u_\dl\rightarrow\r\u \textrm{  in   }C([0,T];L_{weak}^\f{2\gamma}{\gamma+1}(\O)).
\end{equation}
Similarly, we have, as $\dl\to 0$,
$$\H_\dl\rightarrow\H \textrm{  in   }C([0,T];L_{weak}^2(\O)).$$
Thus, the limits $\r$, $\r \u$, $\H$ satisfy the
initial conditions of \eqref{e4} in the sense of distribution.

Since $\gamma>\f{3}{2}$, \eqref{x4} and \eqref{m21} combined with the compactness of $H^1(\O)\hookrightarrow L^2(\O)$ imply, as $\dl\to 0$,
$$\r_\dl\u_\dl\otimes\u_\dl\rightarrow\r\u\otimes\u   \textrm{ in } \mathcal{D}'(\O\times (0,T)).$$
Consequently, letting $\dl\rightarrow 0$ in \eqref{m2} and making use of \eqref{m20}-\eqref{x4}, $(\r, \u, \H)$ satisfies
\begin{equation}\label{m26}
\partial _t \r+\Dv(\r \u)=0,
\end{equation}
%in $\mathcal{D}'(\R^3\times(0,T))$,
\begin{equation}\label{ml27}
\partial _t(\r \u)+\Dv(\r \u\otimes \u)-\mu \triangle \u -(\lambda+\mu)\nabla \Dv
\u +a \nabla \ov{\r^\gamma}=(\na \times \H)\times \H,
\end{equation}
\begin{equation}\label{m28}
\H_t-\nabla\times(\u\times\H)=-\nabla\times(\nu\nabla\times\H), \quad \Dv\H=0,
\end{equation}
 in $\mathcal{D}'(\O\times (0,T))$. Therefore the only thing left to complete
the proof of Theorem \ref{mt} is to show the strong convergence of
$\r_\dl$ in $L^1$ or, equivalently,
$\ov{\r^\gamma}=\r^\gamma$.

Since $\r_\dl$, $\u_\dl$ is a
renormalized solution of the continuity equation \eqref{m2} in
$\mathcal{D}'(\R^3\times(0,T))$, we have
\begin{equation}\label{m35}
T_k(\r_\dl)_t+\Dv(T_k(\r_\dl)\u_\dl)+(T'_k(\r_\dl)\r_\dl-T_k(\r_\dl))\Dv
\u_\dl=0 \textrm{ in }\mathcal{D}'(\R^3\times(0,T)),
\end{equation}
where $T_k$ is the cut-off functions defined as
follows:
$$T_k(z)=kT\left(\frac{z}{k}\right)\; \textrm{ for } z\in R, \;
k=1,2, \dots $$ and $T\in C^\infty(R)$ is concave and is chosen such that
\begin{equation*}
T(z)=\begin{cases} z,& z\leq 1, \\2, & z\geq 3. \end{cases}
\end{equation*}
Passing to the limit for $\dl\rightarrow 0+$, we obtain
\begin{equation*}
\partial_t\ov{T_k(\r)}+\Dv(\ov{T_k(\r)}\u)+\ov{(T'_k(\r)\r
-T_k(\r))\Dv \u}=0 \textrm{ in }
\mathcal{D}'((0,T)\times \R^3),
\end{equation*}
where
\begin{equation*}
(T'_k(\r_\dl)\r_\dl-T_k(\r_\dl))\Dv\u_\dl\rightarrow \ov{(T'_k(\r)\r-T_k(\r))\Dv \u} \textrm{ weakly in } L^2(\O\times (0,T))),
\end{equation*}
and \begin{equation*}
T_k(\r_\dl)\rightarrow\ov{T_k(\r)} \textrm{ in }
C([0,T];L^p_{weak}(\O)) \textrm{ for all }  1\leq p<\infty.
\end{equation*}

\subsection{The effective viscous flux}
In this section, we discuss the effective viscous flux $p(\r)-(\lambda+2\mu)\Dv
\u$. Similarly to
\cite{p2, f1, f2}, we prove the following auxiliary result:

\begin{Lemma}\label{m31}
Let $\r_\dl$, $\u_\dl$ be the
sequence of approximation solutions obtained in Proposition \eqref{m2}.
Then,
$$
\lim_{\dl\rightarrow
0+}\int_0^T\psi\int_{\O}\phi(a\r_\dl^\gamma-(\lambda+2\mu)\Dv
\u_\dl)T_k(\r_\dl)\,\mathrm{d}x\mathrm{d}t$$
$$=\int_0^T\psi\int_{\O}\phi(a\ov{\r^\gamma}-(\lambda+2\mu)\Dv
\u)\ov{T_k(\r)}\,\mathrm{d}x\mathrm{d}t, $$ for
any $\psi\in \mathcal{D}(0,T)$ and $\phi\in \mathcal{D}(\O)$.
\end{Lemma}

\begin{proof}
As in \cite{f1, f2}, we consider the
operators$$\mathcal{A}_i[v]=\triangle^{-1}[\partial_{x_i}v],\;
i=1,2,3$$where $\triangle^{-1}$ stands for the inverse of the
Laplace operator on $\R^3$. To be more specific, $\mathcal{A}_i$ can
be expressed by their Fourier symbol
\begin{equation*}
\mathcal{A}_i[\cdot]=\mathcal{F}^{-1}\left[\f{-\mathrm{i}\xi_i}{|\xi|^2}\mathcal{F}[\cdot]\right],
\; i=1,2,3,
\end{equation*}
with the following properties (see \cite{f1}):
\begin{equation*}
\begin{split}
&\| \mathcal{A}_i v\|_{W^{1,s}(\O)}\leq c(s,\O)\|
v\|_{L^s(R^3)}, \; 1<s<\infty, \\
&\| \mathcal{A}_i v\|_{L^q(\O)}\leq c(q,s,\O)\|
v\|_{L^s(R^3)}, \; q \textrm{ finite, provided } \f{1}{q}\geq
\f{1}{s}-\f{1}{3},\\
&\| \mathcal{A}_i v\|_{L^{\infty}(\O)}\leq c(s,\O)\|
v\|_{L^s(R^3)},\; \textrm{ if } s>3.
\end{split}
\end{equation*}
Next, we use the
quantities$$\varphi_i(t,x)=\psi(t)\phi(x)\mathcal{A}_i[T_k(\r_\dl)], \; \psi\in
\mathcal{D}(0,T), \; \phi\in \mathcal{D}(\O), \; i=1,2,3, $$ as the test
functions for the momentum balance equation in \eqref{m2} to obtain,
\begin{equation}\label{m29}
\begin{split}
&\int_0^T\!\!\!\!\int_{\O}\psi\phi(a\r_\dl^\gamma+\dl\r_\dl^\beta-(\lambda+2\mu)\Dv
\u_\dl)T_k(\r_\dl)\,\mathrm{d}x\mathrm{d}t \\
&=\int_0^T\!\!\!\!\int_{\O}\psi\partial_{x_i}\phi((\lambda+\mu)\Dv
\u_\dl-a\r_\dl^\gamma-\dl\r_\dl^\beta)\mathcal{A}_i
[T_k(\r_\dl)]\,\mathrm{d}x\mathrm{d}t \\
&\quad +\mu\int_0^T\!\!\!\!\int_{\O}\psi\left(\partial_{x_j}\phi
\partial_{x_j}u^i_\dl
\mathcal{A}_i[T_k(\r_\dl)]-u^i_\dl\partial_{x_j}\phi\partial_{x_j}\mathcal{A}_i
[T_k(\r_\dl)]+u^i_\dl\partial
_{x_i}\phi T_k(\r_\dl)\right)\,\mathrm{d}x\mathrm{d}t\\
&\quad -\int_0^T\!\!\!\!\int_{\O}\phi\rho_\dl
u_\dl^i\left(\partial_t\psi
\mathcal{A}_i[T_k(\r_\dl)]+\psi
\mathcal{A}_i[(T_k(\r_\dl)-T'_k(\r_\dl)\r_\dl)\Dv
\u_\dl]\right)\,\mathrm{d}x\mathrm{d}t\\
&\quad -\int_0^T\!\!\!\!\int_{\O}\psi\r_\dl u^i_\dl
u_\dl^j\partial_{x_j}\phi
\mathcal{A}_i[T_k(\r_\dl)]\,\mathrm{d}x\mathrm{d}t\\
&\quad +\int_0^T\!\!\!\!\int_{\O}\psi
u_\dl^i\left(T_k(\r_\dl)\mathcal{R}_{i,j}[\r_\dl
u^j_\dl]-\phi \r_\dl u^j_\dl
\mathcal{R}_{i,j}[T_k(\r_\dl)]\right)\,\mathrm{d}x\mathrm{d}t\\
&\quad -\int_0^T\!\!\!\!\int_{\O}\psi\phi(\na \times \H_\dl)\times \H_\dl\cdot\mathcal{A}[T_k(\r_\dl)]\,\mathrm{d}x\mathrm{d}t,
\end{split}
\end{equation}
where the operators
$\mathcal{R}_{i,j}=\partial_{x_j}\mathcal{A}_i[v]$ and the summation
convention is used to simplify notations.

Analogously, we can repeat the above arguments for
equation \eqref{ml27} and the test functions
$$\varphi_i(t,x)=\psi(t)\phi(x)\mathcal{A}_i[\overline{T_k(\rho)}],\; i=1,2,3,$$
to obtain
\begin{equation}\label{m30}
\begin{split}
&\int_0^T\!\!\!\!\int_{\O}\psi\phi(a\ov{\r^\gamma}-(\lambda+2\mu)\Dv
\u)\ov{T_k(\r)}\,\mathrm{d}x\mathrm{d}t\\
&=\int_0^T\!\!\!\!\int_{\O}\psi\partial_{x_i}\phi((\lambda+\mu)\Dv
\u-a\ov{\r^\gamma})\mathcal{A}_i[\ov{T_k(\r)}]\,\mathrm{d}x\mathrm{d}t\\
&\quad +\mu\int_0^T\!\!\!\!\int_{\O}\psi\left(\partial_{x_j}\phi
\partial_{x_j}u^i
\mathcal{A}_i[\ov{T_k(\r)}]-u^i\partial_{x_j}\phi\partial_{x_j}\mathcal{A}_i[\ov{T_k(\r)}]+u^i\partial_{x_i}\phi
\ov{T_k(\r)}\right)\,\mathrm{d}x\mathrm{d}t\\
&\quad -\int_0^T\!\!\!\!\int_{\O}\phi\r u^i\left(\partial_t\psi
\mathcal{A}_i[\ov{T_k(\r)}]+\psi
\mathcal{A}_i[\ov{(T_k(\r)-T'_k(\r)\r)\Dv
\u}]\right)\,\mathrm{d}x\mathrm{d}t\\
&\quad -\int_0^T\!\!\!\!\int_{\O}\psi\rho u^i u^j\partial_{x_j}\phi
\mathcal{A}_i[\ov{T_k(\r)}]\,\mathrm{d}x\mathrm{d}t\\
&\quad +\int_0^T\!\!\!\!\int_{\O}\psi
u^i\left(\ov{T_k(\r)}\mathcal{R}_{i,j}[\phi\r u^j]-\phi \r
u^j \mathcal{R}_{i,j}[\ov{T_k(\r)}]\right)\,\mathrm{d}x\mathrm{d}t\\
&\quad -\int_0^T\!\!\!\!\int_{\O}\psi\phi(\na \times \H)\times \H\cdot\mathcal{A}[\ov{T_k(\r)}]\,\mathrm{d}x\mathrm{d}t.
\end{split}
\end{equation}

Similarly to \cite{f1, f2}, it can be shown that all the terms on
the right-hand side of \eqref{m29} converge to their counterparts in
\eqref{m30}. Indeed, with the relations \eqref{m20}-\eqref{x4}  and the Sobolev embedding theorem
in mind, it is easy to see that it is enough to show
\begin{equation*}
\begin{split}
&\int_0^T\!\!\!\!\int_{\O}\psi
u_\dl^i\left(T_k(\r_\dl)\mathcal{R}_{i,j}[\phi\r_\dl
u^j_\dl]-\phi \r_\dl u^j_\dl
\mathcal{R}_{i,j}[T_k(\r_\dl)]\right)\,\mathrm{d}x\mathrm{d}t\\
&\rightarrow \int_0^T\!\!\!\!\int_{\O}\psi u^i\left(\ov{T_k(\r)}R_{i,j}[\phi\r
u^j]-\phi \r u^j
\mathcal{R}_{i,j}[\ov{T_k(\r)}]\right)\,\mathrm{d}x\mathrm{d}t,
\end{split}
\end{equation*}
because the properties of  $\mathcal{A}_i$ and the weak convergence of
$\u$ in $L^2([0,T]; H^1(\O))$ imply
\begin{equation*}
\mathcal{A}_i(T_k(\r_\dl))\rightarrow
\mathcal{A}_i(\ov{T_k(\r)}) \textrm{ in }
C(\ov{(0,T)\times \O}),
\end{equation*}
\begin{equation*}
\mathcal{R}_{i,j}(T_k(\r_\dl))\rightarrow
\mathcal{R}_{i,j}(\ov{T_k(\r)}) \textrm{ weakly in }
L^p([0,T]\times\O) \textrm{ for all } 1<p<\infty,
\end{equation*}
and
\begin{equation*}
\begin{split}
\mathcal{A}_i[(T_k(\r_\dl)-T'_k(\r)\r)\Dv
\u_\dl]\rightarrow
\mathcal{A}_i[\ov{(T_k(\r)-T'_k(\r)\r)\Dv \u}] \; \textrm{ weakly in } L^2([0,T]; H^1(\O)).
\end{split}
\end{equation*}
From Lemma 3.4 in \cite{f1}, we have
\begin{equation*}
\begin{split}
& T_k(\r_\dl)\mathcal{R}_{i,j}[\phi\r_\dl
u^j_\dl]-\phi \r_\dl u^j_\dl
\mathcal{R}_{i,j}[T_k(\r_\dl)]\\
&\rightarrow \overline{T_k(\r)}R_{i,j}[\phi\r u^j]-\phi \r u^j
\mathcal{R}_{i,j}[\ov{T_k(\r)}] \; \textrm{ weakly in }\;
L^r(\O), \; i, j=1,2,3,
\end{split}
\end{equation*}
for some $r>1$.
Hence, we complete the proof of Lemma \ref{m31}.
\end{proof}

\subsection{ The amplitude of oscillations}
The main result of this subsection reads as follows, and is essentially taken  from \cite{f1} (cf. Lemma 4.3 in \cite{f1}):

\begin{Lemma}\label{m34}
There exists a constant c independent of k such that
$$
\limsup_{\dl\rightarrow 0+}\|
T_k(\r_\dl)-T_k(\r)\|_{L^{\gamma+1}((0,T)\times
\O)}\leq c.
$$
\end{Lemma}

\begin{proof}
By the convexity of functions $t\rightarrow p(t)$, $t\rightarrow
-T_k(t)$, one has
\begin{equation}\label{m32}
\begin{split}
&\limsup_{\dl\rightarrow
0+}\int_0^T\!\!\!\!\int_{\O}\left(\r_\dl^\gamma
T_k(\r_\dl)-\ov{\r^\gamma}
(\ov{T_k(\r)})\right)\mathrm{d}x\mathrm{d}t\\
&=\limsup_{\dl\rightarrow
0+}\int_0^T\!\!\!\!\int_{\O}(\r_\dl^\gamma-\r^\gamma)(T_k(\r_\dl)-T_k(\r))\,\mathrm{d}x\mathrm{d}t
\\
&\quad +\int_0^T\!\!\!\!\int_{\O}(\ov{\r^\gamma}-\r^\gamma)(T_k(\r)-\overline{T_k(\r)})
\,\mathrm{d}x\mathrm{d}t\\
&\geq \limsup_{\dl\rightarrow
0+}\int_0^T\!\!\!\!\int_{\O}(\r_\dl^\gamma-\r^\gamma)(T_k(\r_\dl)-T_k(\r))
\,\mathrm{d}x\mathrm{d}t.
\end{split}
\end{equation}
On one hand, we have
\begin{equation*}
y^\gamma-z^\gamma=\int_z^y \gamma s^{\gamma-1}\mathrm{d}s\geq
\gamma\int_z^y(s-z)^{\gamma-1}\mathrm{d}s=\gamma(y-z)^\gamma,
\end{equation*}
for all $y\geq z\geq 0$, and
\begin{equation*}
|T_k(y)-T_k(z)|^\gamma\leq |y-z|^\gamma,
\end{equation*}
thus,
\begin{equation*}
\begin{split}
(z^\gamma-y^\gamma)(T_k(z)-T_k(y))
&\geq \gamma|T_k(z)-T_k(y)|^\gamma|T_k(z)-T_k(y)| \\
&=\gamma|T_k(z)-T_k(y)|^{\gamma+1},
\end{split}
\end{equation*}
for all $z,y\geq 0$.
On the other hand,
\begin{equation}\label{m33}
\begin{split}
&\limsup_{\dl\rightarrow 0+}\int_0^T\!\!\!\!\int_{\O}\left(\Dv\u_\dl T_k(\r_\dl)-\Dv
\u \ov{T_k(\r)}\right)\mathrm{d}x\mathrm{d}t\\
&=\limsup_{\dl\rightarrow 0+}\int_0^T\!\!\!\!\int_{\O}\left(T_k(\r_\dl)-T_k(\r)+T_k(\r)-
\ov{T_k(\r)}\right)\Dv\u_\dl\,\mathrm{d}x\mathrm{d}t \\
&\leq 2\sup_{\dl}\| \Dv
u_\dl\|_{L^2((0,T)\times
\O)}\limsup_{\dl\rightarrow 0+}\|
T_k(\r_\dl)-T_k(\r)\|_{L^2((0,T)\times \O)}\\
&\leq c\limsup_{\dl\rightarrow 0+}\|
T_k(\r_\dl)-T_k(\r)\|_{L^2((0,T)\times \O)}\\
&\leq c+\f{1}{2}\limsup_{\dl\rightarrow 0+}\|
T_k(\r_\dl)-T_k(\r)\|^{\gamma+1}_{L^{\gamma+1}((0,T)\times
\O)}.
\end{split}
\end{equation}
The relations \eqref{m32}, \eqref{m33} combined with Lemma \ref{m31}  yield the
desired conclusion.
\end{proof}

\subsection{ The renormalized solutions}
We now use  Lemma \ref{m34} to prove the following crucial result:
\begin{Lemma}\label{m40}
The limit functions $\r, \u$ solve \eqref{m26} in the
sense of renormalized solutions, i.e.,
\begin{equation}\label{m38}
\partial_t b(\r)+\Dv(b(\r)\u)+(b'(\r)\r
-b(\r))\Dv \u=0,
\end{equation}
holds in $\mathcal{D}'(\R^3\times (0,T))$ for any $b \in C^1(R)$
satisfying $b'(z)=0$ for all $z\in \textrm{R}$ large enough, say,
$z\geq M $, where the constant M may depend on $b$.
\end{Lemma}
\begin{proof}
Regularizing \eqref{m35}, one gets
\begin{equation}\label{m36}
\partial_t S_m[\overline{T_k(\rho)}]+\textrm{div}(S_m[\overline{T_k(\rho)}]\textrm{\textbf{u}})+S_m[\overline{(T'_k(\rho)\rho
-T_k(\rho))\textrm{div} \textbf{\textrm{u}}}]=r_m,
\end{equation}
where $S_m[v]=v_m*v$ are the standard smoothing operators and
$r_m\rightarrow 0$ in $L^2([0,T]; L^2(\R^3))$ for any fixed $k$ (see
Lemma 2.3 in \cite{p1}). Now, we are allowed to multiply \eqref{m36} by
$b'(S_m[\ov{T_k(\r)}])$. Letting $m\rightarrow \infty$, we obtain
\begin{equation}\label{m39}
\begin{split}
&\partial_t b[\ov{T_k(\r)}]+\Dv(b[\ov{T_k(\r)}]\u)+(b'(\ov{T_k(\r)})\ov{T_k(\r)}
-b(\ov{T_k(\r)}))\Dv \u\\
&=b'(\ov{T_k(\r)})[\ov{(T'_k(\r)\r
-T_k(\r))\Dv \u}] \textrm{ in }
\mathcal{D}'((0,T)\times\R^3).
\end{split}
\end{equation}
At this stage, the main idea is to let $k\rightarrow \infty$ in
\eqref{m39}. We have
\begin{equation*}
\ov{T_k(\r)}\rightarrow \r \; \textrm{ in } L^p(\O\times (0,T))\textrm{ for
any } 1\leq p<\gamma, \quad \textrm{ as } k\rightarrow\infty,
\end{equation*}
since
\begin{equation*}
\|\ov{T_k(\r)}-\r\|_{L^p(\O\times (0,T))}\leq
\liminf_{\dl\rightarrow 0+}\|
T_k(\r_\dl)-\r_\dl\|_{L^p(\O\times (0,T))},
\end{equation*}
and
\begin{equation}\label{m41}
\|T_k(\r_\dl)-\rho_\dl\|^p_{L^p(\O\times (0,T))}\leq
2^pk^{p-\gamma}\|\r_\dl\|^{\gamma}_{L^{\gamma}(\O\times (0,T))}\leq ck^{p-\gamma}.
\end{equation}

Thus \eqref{m39} will imply \eqref{m38} provided we show
\begin{equation*}
b'(\ov{T_k(\r)})[\ov{(T'_k(\r)\r
-T_k(\r))\Dv \u}]\rightarrow 0 \textrm{
in } L^1(\O\times (0,T)) \textrm{ as } k\rightarrow\infty.
\end{equation*}
To this end, let us denote
$$Q_{k,M}=\{(t,x)\in \O\times (0,T) \mid \ov{T_k(\r)}\leq M\},$$
%where M is the constant related to b by Lemma \ref{m40}, One has
then
\begin{equation*}
\begin{split}
&\int_0^T\!\!\!\!\int_{\O}\left|b'(\ov{T_k(\r)})[\ov{(T'_k(\r)\r
-T_k(\r))\Dv \u}]\right|\,\mathrm{d}x\mathrm{d}t\\
&\leq\sup_{0\leq z\leq
M}|b'(z)|\iint_{Q_{K,M}}\left|\ov{(T'_k(\r)\r
-T_k(\r))\Dv \u}\right|\,\mathrm{d}x\mathrm{d}t\\
&\leq\sup_{0\leq z\leq M}|b'(z)|\liminf_{\dl\rightarrow
0+}\|(T'_k(\r_\dl)\r_\dl
-T_k(\r_\dl))\Dv
\u_\dl\|_{L^1(Q_{k,M})}\\
&\leq\sup_{0\leq z\leq M}|b'(z)|\sup_{\dl}\|
\u_\dl\|_{L^2([0,T];H^1(\O))}\liminf_{\dl
\rightarrow 0+}\| T'_k(\r_\dl)\rho_\dl
-T_k(\r_\dl)\|_{L^2(Q_{k,M})}\\
&\leq c\liminf_{\dl \rightarrow 0+}\|
T'_k(\r_\dl)\r_\dl
-T_k(\r_\dl)\|_{L^2(Q_{k,M})}.
\end{split}
\end{equation*}
Now, by interpolation, one has
\begin{equation}\label{m42}
\begin{split}
&\| T'_k(\r_\dl)\r_\dl
-T_k(\r_\dl)\|_{L^2(Q_{k,M})}^2 \\
&\leq \| T'_k(\r_\dl)\r_\dl
-T_k(\r_\dl)\|_{L^1(\O\times (0,T))}^{\f{\gamma-1}{\gamma}}\|
T'_k(\r_\dl)\r_\dl
-T_k(\r_\dl)\|_{L^{\gamma+1}(Q_{k,M})}^{\f{\gamma+1}{\gamma}}.
\end{split}
\end{equation}
Similarly to \eqref{m41}, we have
\begin{equation*}
\| T'_k(\r_\dl)\r_\dl
-T_k(\r_\dl)\|_{L^1(\O\times (0,T))}\leq
ck^{1-\gamma}\sup_{\dl}\|\r_\dl\|^{\gamma}_{L^{\gamma}}\leq
ck^{1-\gamma},
\end{equation*}
and, using $ T'_k(z)z\leq T_k(z)$,
\begin{equation}\label{m43}
\begin{split}
&\frac12\|T'_k(\r_\dl)\r_\dl
-T_k(\r_\dl)\|_{L^{\gamma+1}(Q_{k,M})}\\
&\leq \| T_k(\r_\dl)
-T_k(\r)\|_{L^{\gamma+1}(\O\times (0,T))}+\|
T_k(\r)\|_{L^{\gamma+1}(Q_{k,M})}\\
&\leq \| T_k(\r_\dl)
-T_k(\r)\|_{L^{\gamma+1}(\O\times (0,T))}+\|
\ov{T_k(\r)}\|_{L^{\gamma+1}(Q_{k,M})} \\
&\qquad +\|
\ov{T_k(\r)}
-T_k(\r)\|_{L^{\gamma+1}(\O\times (0,T))}\\
&\leq \| T_k(\r_\dl)
-T_k(\r)\|_{L^{\gamma+1}(\O\times (0,T))}+M c(\O)\\
&\qquad +\|\ov{T_k(\r)}-T_k(\r)\|_{L^{\gamma+1}(\O\times (0,T))}.
\end{split}
\end{equation}
From Lemma \ref{m34} and \eqref{m43}, we obtain
\begin{equation*}
\limsup_{\dl\rightarrow 0+}\| T'_k(\r_\dl)\r_\dl
-T_k(\r_\dl)\|_{L^{\gamma+1}(Q_{k,M})}\leq 4c+2M c(\O),
\end{equation*}
which, together with \eqref{m42}-\eqref{m43}, completes the proof of Lemma \ref{m40}.
\end{proof}

\subsection{ Strong convergence of the density}
Now, we can complete the proof of Theorem \ref{mt}. To this end, we
introduce a sequence of functions $L_k\in C^1(R)$:
\begin{equation*}
L_k(z)=\begin{cases}z\textrm{ln}z,&  0\leq z<k,\\
z\textrm{ln}(k)+z\int_k^z\frac{T_k(s)}{s^2}\mathrm{d}s, & z\geq
k.\end{cases}
\end{equation*}
Noting that $L_k$ can be written as
\begin{equation}\label{ml51}
L_k(z)=\beta_k z+b_k(z),
\end{equation}
where $b_k$ satisfies the conditions  in Lemma \ref{m40}, we can use
the fact that $\rho_\dl, \u_\dl$
are renormalized solutions of \eqref{m2} to deduce
\begin{equation}\label{ml45}
\partial_t L_k(\r_\dl)+
\Dv(L_k(\r_\dl)\u_\dl)+T_k(\r_\dl)\Dv
\u_\dl=0.
\end{equation}
Similarly, by \eqref{m26} and Lemma \ref{m40}, we have
\begin{equation}\label{m46}
\partial_t L_k(\rho)+
\textrm{div}(L_k(\rho)\textbf{\textrm{u}})+T_k(\rho)\textrm{div}
\textbf{\textrm{u}}=0,
\end{equation}
in $\mathcal{D}'((0,T)\times \Omega)$. By \eqref{ml45}, we can assume, as $\dl\to 0$,
\begin{equation*}
L_k(\r_\dl)\rightarrow \ov{L_k(\r)} \textrm{ in }
C([0,T];L^{\gamma}_{weak}(\O)).
\end{equation*}

Taking the difference of \eqref{ml45} and \eqref{m46} and integrating with
respect to t,  we get
\begin{equation}\label{m47}
\begin{split}
&\int_{\O}(L_k(\r_\dl)-L_k(\r))\phi\,\mathrm{d}x\\
&=\int_0^t\int_{\O}\left((L_k(\r_\dl)\u_\dl-L_k(\r)\u)\cdot\nabla
\phi+(T_k(\r)\Dv\u-T_k(\r_\dl)\Dv\u_\dl)\phi\right)\mathrm{d}x\mathrm{d}t,
\end{split}
\end{equation}
for any $\phi\in \mathcal{D}(\O)$. Passing to the limit for
$\dl\rightarrow 0$ and making use of \eqref{m47}, one obtains
\begin{equation}\label{m50}
\begin{split}
&\int_{\O}(\ov{L_k(\r)}-L_k(\r))\phi\,\mathrm{d}x\\
&=
\int_0^t\int_{\O}(\ov{L_k(\r)}
-L_k(\r))\u\cdot\nabla\phi\,\mathrm{d}x\mathrm{d}t\\
&\quad +\lim_{\dl\rightarrow 0+}\int_0^t\int_{\O}(T_k(\r)\Dv
\u-T_k(\r_\dl)\Dv\u_\dl)\phi\,\mathrm{d}x\mathrm{d}t,
\end{split}
\end{equation}
for any $\phi\in \mathcal{D}(\Omega)$.

Since the velocity components $u^i, i=1,2,3, $ belong to $L^2([0,T];
W^{1,2}_0(\O))$, one has the following (see Theorem 4.2 in \cite{f2}):
\begin{equation*}
\f{|\u|}{\textrm{dist}[x,\partial\O]}\in
L^2([0,T]; L^2(\O)).
\end{equation*}
Let us consider a sequence of functions $\phi_m\in
\mathcal{D}(\O)$ which  approximate the characteristic
function of $\O$ such that
\begin{equation}\label{aa1}
\begin{split}
&0\leq\phi_m\leq 1, \quad \phi_m(x)=1 \textrm{ for all x such that
dist}[x,\partial\Omega]\geq\f{1}{m},  \\
&\textrm{and} \; |\nabla\phi_m(x)|\leq 2m \textrm{ for all } x\in \O.
\end{split}
\end{equation}
Taking the sequence $\phi=\phi_m$ as the test functions in \eqref{m50},
making use of the boundary conditions in \eqref{e4}, and passing to the
limit as $m\rightarrow \infty$, one has
\begin{equation}\label{m52}
\int_{\O}(\ov{L_k(\r)}-L_k(\r))\,\mathrm{d}x=
\int_0^t\!\!\!\int_{\O}T_k(\r)\Dv \u
\,\mathrm{d}x\mathrm{d}t-\lim_{\dl\rightarrow
0+}\int_0^t\!\!\!\int_{\O}T_k(\r_\dl)\Dv
\u_\dl\,\mathrm{d}x\mathrm{d}t.
\end{equation}
We observe that the term $\overline{L_k(\r)}-L_k(\r)$ is bounded by
\eqref{ml51}.

At this stage, the main idea is to let $k\rightarrow\infty$ in
\eqref{m52}. By  \eqref{ml3}, we can assume
\begin{equation*}
\rho_\varepsilon
\textrm{ln}(\rho_\varepsilon)\rightarrow\overline{\rho
\textrm{ln}(\rho)} \textrm{ weakly star in }
L^{\infty}([0,T];L^\alpha(\Omega)) \textrm{ for all }
1\leq\alpha<\gamma.
\end{equation*}
We also have
\begin{equation*}
\ov{L_k(\r)}\rightarrow \ov{\r \textrm{ln}(\r)}
\textrm{ in } L^\infty([0,T]; L^\alpha(\O)) \textrm{ as }
k\rightarrow\infty\textrm{ for all } 1\leq\alpha<\gamma,
\end{equation*}
since, by \eqref{ml3},
\begin{equation*}
\lim_{k\rightarrow\infty}r(k)=0, \quad \textrm{ where } r(k):=
\text{meas}\{(x,t)\in\O\times (0,T) |\rho_\dl(x,t)\geq k\};
\end{equation*}
and because $L_k(z)\leq z\textrm{ln}z$, repeating the similar procedure to \eqref{m18},
we have
\begin{equation*}
\begin{split}
&\|\ov{L_k(\r)}-\ov{\r\textrm{ln}(\r)}\|_{L^\infty([0,T]; L^\alpha(\O))}\\
&\leq\sup_{t\in[0,T]}\liminf_{\dl\rightarrow 0+}\| L_k(\r_\dl)-\r_\dl
\textrm{ln}(\r_\dl)\|_{L^\infty([0,T];L^\alpha(\O))}\\
&\leq 2q(k)\sup_{\dl}\sup_{t\in[0,T]}\max\{1, \int_{\O}M(\r_\dl^\alpha
|\textrm{ln}\r_\dl|^\alpha)\,\mathrm{d}x\}\\
&\leq 2q(k)\sup_{\dl}\sup_{t\in[0,T]}\max\{1, 2\int_{\O}(1+\r_\dl^\alpha
|\textrm{ln}\r_\dl|^\alpha)\textrm{ln}(1+\r_\dl^\alpha
|\textrm{ln}\r_\dl|^\alpha)\,\mathrm{d}x\}\\
&\leq 2q(k)\sup_{\dl}\sup_{t\in[0,T]}\max\{1, c(\alpha)\textrm{meas}\{\O\}+c(\alpha)\int_{\O\cap\{\r_\dl\geq e\}}\r_\dl^\alpha
|\textrm{ln}\r_\dl|^{\alpha+1}\,\mathrm{d}x\}\\
&\leq 2q(k)\sup_{\dl}\sup_{t\in[0,T]}\max\{1, c(\alpha)\textrm{meas}\{\O\}+c(\alpha, \gamma)\int_{\O}\r_\dl^\gamma
\,\mathrm{d}x\}\\
&\leq cq(k)\rightarrow 0, \; \textrm{ as } k\rightarrow\infty,
\end{split}
\end{equation*}
where the function M is defined in \eqref{l1}, c is a constant independent of $\dl$ and
\begin{equation*}
q(k):=\|\chi_{[\r_\dl\geq
k]}\|_{L_N(\O)}\leq
\left(N^{-1}\left(\f{1}{r(k)}\right)\right)^{-1}.
%, \textrm{ and }\lim_{k\rightarrow \infty}q(k)=0
\end{equation*}
Similarly, we have
\begin{equation*}
L_k(\r)\rightarrow \r \textrm{ln}(\r)\textrm{ in }
L^\infty([0,T]; L^\alpha(\O)) \textrm{ as }
k\rightarrow\infty, \; \textrm{ for all } 1\leq\alpha<\gamma,
\end{equation*}
and, by Lemma \ref{m34},
\begin{equation}\label{m61}
T_k(\r)\rightarrow \ov{T_k(\r)}\textrm{ in } L^\alpha([0,T];
L^\alpha(\O)) \textrm{ as } k\rightarrow\infty, \; \textrm{ for all }
1\leq\alpha<\gamma+1.
\end{equation}

Finally, making use of Lemma \ref{m31} and the monotonicity of the
pressure (see \eqref{m32}), we obtain the following estimate on the right hand side of \eqref{m52}:
\begin{equation}\label{m62}
\begin{split}
&\int_0^t\int_{\O}T_k(\r)\Dv \u
\,\mathrm{d}x\mathrm{d}t-\lim_{\dl\rightarrow
0+}\int_0^t\int_{\O}T_k(\r_\dl)\Dv
\u_\dl\,\mathrm{d}x\mathrm{d}t \\
&\leq \int_0^t\int_{\O}(T_k(\r)-\ov{T_k(\r)})\Dv
\u\,\mathrm{d}x\mathrm{d}t.
\end{split}
\end{equation}
From \eqref{m61} and the Sobolev embedding theorem,  we see that the
right hand side of \eqref{m62} tends to zero as $k\rightarrow\infty$.
Accordingly, one can pass to the limit for $k\rightarrow\infty$ in
\eqref{m52} to conclude
\begin{equation}\label{m71}
\int_{\Omega}\left(\ov{\r \textrm{ln}(\r)}-\r
\textrm{ln}(\r)\right)(x,t)\,\mathrm{d}x=0,  \;   \text{for} \; t\in [0,T] \; a.e.
\end{equation}
Because of the convexity of the function $z\rightarrow
z\textrm{ln}z$, we have
\begin{equation*}
\ov{\r \textrm{ln}(\r)}\geq\r \textrm{ln}(\r), \;  \textrm{
a.e.} \textrm{ in } \O\times (0,T),
\end{equation*}
which, combining with \eqref{m71}, implies
\begin{equation}\label{m72}
\ov{\r\textrm{ln}(\r)}(t)=\r \textrm{ln}(\r)(t),\;
\textrm{ for  } t\in [0,T] \; a.e.
\end{equation}
Theorem 2.11 in \cite{f2}, combined with \eqref{m72}, implies
\begin{equation*}
\rho_\varepsilon\rightarrow\rho, \; \textrm{ a.e. in } \O\times (0,T).
\end{equation*}
From the estimate \eqref{ml3} on $\r$, together with Proposition 2.1 in \cite{f2},  again we know,
\begin{equation*}
\r_\dl\rightarrow\r \textrm{ weakly in }
L^1(\O\times (0,T)),
\end{equation*}
subject to a subsequence.
By  Theorem 2.10 in \cite{f2}, we know that for any
$\eta>0$, there exists $\sigma>0$ such that for all $\dl>0$,
\begin{equation*}
\int_{E}\r_\dl(t,x)\,\mathrm{d}x\mathrm{d}t<\eta,
\end{equation*}
for any measurable set $E\subset \O\times (0,T)$ with
$\textrm{meas}\{E\}<\sigma$.

On the other hand, by virtue of Egorov's Theorem, for $\sigma>0$
given above, there exists a measurable set $E_\sigma\subset
\O\times (0,T)$ such that
\begin{equation*}
\textrm{meas}\{E_\sigma\}<\sigma,\textrm{ and }
\r_\dl(x,t)\rightarrow \r(x,t)\textrm{ uniformly in }
\O\times (0,T)-E_\sigma.
\end{equation*}
Therefore, we have
\begin{equation}\label{m80}
\begin{split}
&\iint_{\O\times (0,T)}|\r_\dl-\rho|\,\mathrm{d}x\mathrm{d}t\\
&\leq
\iint_{E_\sigma}|\r_\dl-\r|\,\mathrm{d}x\mathrm{d}t+\iint_{\O\times (0,T)-E_\sigma}
|\r_\dl-\r|\,\mathrm{d}x\mathrm{d}t\\
&\leq 2\eta+T\textrm{meas}\{\O\}\sup_{(x,t)\in
E^{c}_\sigma}|(\r_\dl-\r)(x,t)|,
\end{split}
\end{equation}
which tends to zero if we first let $\dl\rightarrow 0+$, and
then let $\eta\rightarrow 0+$. The strong convergence of the
sequence $\r_\dl$ in $L^1(\O\times (0,T))$ follows from
\eqref{m80}.

The proof of Theorem \ref{mt} is completed.

\bigskip

\section{Large-Time Behavior of Weak Solutions}
Our final goal in this paper is to study the large-time behavior of the finite energy weak solutions, whose existence is ensured by
Theorem \ref{mt}.

First of all, from Theorem \ref{mt}, we have
\begin{equation}\label{x2b}
\ess_{t>0}E(t)+\int_0^\infty\!\!\!\int_\O\left\{\mu|D\u|^2+(\lambda+\mu)(\Dv\u)^2+\nu|\nabla\times\H|^2\right\}\,\mathrm{d}x\mathrm{d}t\\
\leq E(0).
\end{equation}
Following the idea in \cite{f5}, we consider a sequence
\begin{equation*}
\begin{cases}
\r_m(x,t):= \r(x,t+m);\\  %\triangleq
\u_m(x,t):= \u(x,t+m); \\
\H_m(x,t):= \H(x,t+m),
\end{cases}
\end{equation*}
for all integer m, and $t\in(0,1), \;x\in\O$.
It is easy to see that \eqref{x2b} yields uniform bounds of
\begin{equation*}
\r_m\in L^\infty([0,1]; L^\gamma(\O)), \quad \H_m\in L^\infty([0,1]; L^2(\O))
\end{equation*}
\begin{equation*}
\sqrt{\r_m}\u_m\in L^\infty([0, 1]; L^2(\O)), \quad \r_m\u_m\in L^\infty([0,1]; L^{\f{2\gamma}{\gamma+1}}(\O)),
\end{equation*}
which are independent of m.
Moreover, we have
\begin{equation}\label{x3}
\lim_{m\rightarrow \infty}\int_0^{1}\left(\|\nabla\u_m\|^2_{L^2(\O)}+\|\nabla\times
\H_m\|^2_{L^2(\O)}\right)\mathrm{d}t=0.
\end{equation}
Hence, choosing a subsequence if necessary, we can assume that, as $m\to\infty$,
$$\r_m(x,t)\rightarrow\r_s \textrm{ weakly in } L^\gamma(\O\times (0,1));$$
$$\u_m(x,t)\rightarrow\u_s \textrm{ weakly in } L^2([0,1];H_0^1(\O));$$
$$\H_m(x,t)\rightarrow\H_s \textrm{ weakly in } L^2([0,1];H_0^1(\O)).$$
Furthermore,
$$\int_\O\r_s\,\mathrm{d}x\leq \liminf_{m\rightarrow \infty}\int_\O\r_m(t)\,\mathrm{d}x\leq C(E_0).$$
Therefore, from the Poincar\'{e} inequality and \eqref{x3}, we know
$$\lim_{m\rightarrow \infty}\int_0^{1}\|\u_m\|^2_{L^2(\O)}\mathrm{d}t=0.$$
This, combined with the compactness of $H^1\hookrightarrow L^2$, implies
$$\u_s=0, \textrm{ a.e in }\O\times (0,1).$$
Similarly, we know that
\begin{equation}\label{cc1}
\H_s=0, \;\textrm{ a.e in }\O\times (0,1).
\end{equation}

On the other hand, by Sobolev inequality, H\"{o}lder inequality, \eqref{x2b} and \eqref{x3}, we have
\begin{equation}\label{x4b}
\lim_{m\rightarrow \infty}\int_0^{1}\left(\|\r_m|\u_m|^2\|_{L^{\f{3\gamma}{\gamma+3}}(\O)}+\|\r_m|\u_m|\|^2_
{L^{\f{6\gamma}{\gamma+6}}(\O)}\right)\mathrm{d}t=0.
\end{equation}
Since $\r, \u$ are solutions to (1.1) in the sense of renormalized solutions, one has, in particular,
\begin{equation}\label{x5}
\rho_t +\Dv(\r\u)=0\textrm{ in }\mathcal{D}'(\O\times (0,T)).
\end{equation}
Then taking test functions $\varphi(x,t)=\psi(t)\phi(x)$ in \eqref{x5}, where $\psi(t)\in\mathcal{D}(0,1), \phi\in\mathcal{D}(\O)$,
we have, using integrating by parts,
$$\int_0^1\left(\int_\O\r_m\phi\,\mathrm{d}x\right)\psi'(t)\mathrm{d}t
+\int_0^1\int_\O\r_m\u_m\nabla\phi\psi\,\mathrm{d}x\mathrm{d}t=0.$$
Letting $m\rightarrow\infty$ and using \eqref{x4b}, we obtain
$$\int_0^1\left(\int_\O\r_s\phi\,\mathrm{d}x\right)\psi'(t)\mathrm{d}t=0.$$
This implies that $\r_s$ must be independent of t, by the arbitrariness of $\psi$.

Now, following  the procedure used in Section 5, or in Lemma 4.1 in \cite{f1}, one can obtain
$$\r_m^{\gamma+\theta}\textrm{ is bounded in }L^1(\O\times (0,1)),\textrm{ independently of }m>0,$$ for some $\theta>0$.
Consequently, one has
\begin{equation}\label{x6}
\r_m^\gamma\rightarrow \ov{\r^\gamma}\textrm{ weakly in
}L^1(\O\times (0,1)).
\end{equation}
Therefore, passing to the limit in the momentum balance equation of
\eqref{e1} and using \eqref{x3}, \eqref{x4b}, we get
\begin{equation}\label{y1}
\nabla\ov{\r^\gamma}=0\textrm{ in }\mathcal{D}'(\O).
\end{equation}

Now, we show that the convergence in \eqref{x6} is indeed strong. To this end, similarly to \cite{f5}, we consider
$$G(z)=z^\alpha, \quad 0<\alpha<\min\left\{\f{1}{2\gamma}, \f{\theta}{\theta+\gamma}\right\},$$ so that $b(z)=G(z^\gamma)$
may be used in \eqref{e5}.
Consider the vector functions
$$[G(\r_m^\gamma),0,0,0]\textrm{  and  }[\r_m^\gamma,0,0,0]$$
of the time variable $t$ and the spatial coordinates $x$. Using \eqref{e5} and \eqref{x3}, \eqref{x4b}, we get
\begin{equation}\label{x7}
\rm{Div}[G(\r_m^\gamma),0,0,0] \textrm{  is precompact in  } W_{loc}^{-1,q_1}(\O\times (0,1)),
\end{equation}
for some $q_1>1$ small enough.
Similarly, making use of the momentum balance equation in \eqref{e1}, \eqref{x3}, and \eqref{x4b}, we obtain
\begin{equation}\label{x8}
\rm{Curl}[\r_m^\gamma,0,0,0]\textrm{  is precompact in  } W_{loc}^{-1,q_2}(\O\times (0,1)),
\end{equation}
for some $q_2>1$, where
$$\rm{Div}(f_0, f_1, f_2, f_3):=(f_0)_t+\Sigma_{i=1}^3\partial_{x_i}f_i,$$
and
$$\rm{Curl}(f_0, f_1, f_2, f_3):=\partial_i f_j-\partial_j f_i, \quad x_0:=t,\quad i,j=0,...,3.$$
Meanwhile, we can assume
\begin{equation}\label{x9}
G(\r_m^\gamma)\rightarrow\ov{G(\r^\gamma)}\textrm{ weakly in }L^{p_2}(\O\times (0,1)),
\end{equation}
and
\begin{equation}\label{x10}
G(\r_m^\gamma)\r_m^\gamma\rightarrow\ov{G(\r^\gamma)\r^\gamma}\textrm{ weakly in }L^{r}(\O\times (0,1)),
\end{equation}
with
$$p_2=\f{1}{\alpha}, \quad \f{1}{p_1}+\f{1}{p_2}=\f{1}{r}<1.$$
Using the $L^p$-version of the celebrated div-curl lemma (see \cite{zy}), we deduce that from \eqref{x7}-\eqref{x10}
\begin{equation}\label{x11}
\ov{G(\r^\gamma)}\ov{\r^\gamma}=\ov{G(\r^\gamma)\r^\gamma}.
\end{equation}
As G is strictly monotone, \eqref{x11} implies $\ov{G(\r^\gamma)}=G(\ov{\r^\gamma})$. Since $L^{\f{1}{\alpha}}$ is uniformly convex, this yields
strong convergence in \eqref{x6}.
Therefore, we have
$$\r_m\rightarrow\r_s\textrm{  strongly in  }L^\gamma(\O\times (0,1)).$$
This, combined with \eqref{x6} and \eqref{y1}, gives
$\nabla\r_s^\gamma=0$ in the sense of distributions, which implies
that $\r_s$ is independent of the spatial variables.

Finally, by the energy inequality, the energy converges to a finite constant as t goes to infinity:
$$E_\infty:=\ov{\lim_{t\rightarrow\infty}}E(t),$$
and, by \eqref{x4b} and \eqref{cc1},
$$\lim_{m\rightarrow\infty}\int_m^{m+1}\int_\O\r|\u|^2\,\mathrm{d}x=0,$$
$$\lim_{m\rightarrow\infty}\int_m^{m+1}\int_\O|\H|^2\,\mathrm{d}x=0.$$
Thus
\begin{equation*}
\begin{split}
E_\infty&=\ov{\lim_{m\rightarrow\infty}}\int_m^{m+1}\int_\O\left(\f{1}{2}\r\u^2+\f{a}{\gamma-1}
\r^\gamma+\f{1}{2}|\H|^2\right)\,\mathrm{d}x\mathrm{d}t\\&=
\int_\O\f{a}{\gamma-1}\r_s^\gamma\,\mathrm{d}x.
\end{split}
\end{equation*}
Furthermore, using the continuity equation in (1.1), one easily observe that
$$\r(x, t)\rightarrow \r_s\textrm{ weakly in }L^\gamma(\O)\textrm{ as }t\rightarrow \infty.$$
Thus, we have
\begin{equation*}
\begin{split}
E_\infty&=\int_\O\f{a}{\gamma-1}\r_s^\gamma\,\mathrm{d}x\leq \liminf_{t\rightarrow \infty}\int_\O\f{a}{\gamma-1}\r^\gamma\,\mathrm{d}x
\leq \limsup_{t\rightarrow \infty}\int_\O\f{a}{\gamma-1}\r^\gamma\,\mathrm{d}x\\
&\leq\ov{\lim_{t\rightarrow\infty}}\int_\O\left(\f{1}{2}\r\u^2+\f{a}{\gamma-1}
\r^\gamma+\f{1}{2}|\H|^2\right)\mathrm{d}x
=\ov{\lim_{t\rightarrow\infty}}E(t)=E_\infty.
\end{split}
\end{equation*}
This implies
$$\lim_{t\rightarrow\infty}\int_\O\f{a}{\gamma-1}\r^\gamma\,\mathrm{d}x
=\int_\O\f{a}{\gamma-1}\r_s^\gamma\,\mathrm{d}x,$$ and \eqref{xx1}
follows since the space $L^\gamma$ is uniformly convex.

This completes the proof of Theorem \ref{mt2}.

\bigskip\bigskip

\section*{Acknowledgments}

Xianpeng Hu's research was supported in part by the National Science Foundation grant  DMS-0604362.
Dehua Wang's research was supported in part by the National Science
Foundation grants DMS-0244487, DMS-0604362, and the Office of Naval
Research grant N00014-01-1-0446.


\bigskip\bigskip


\begin{thebibliography}{999}

\bibitem{a1} R. A. Admas, \emph{ Sobolev spaces.} Pure and Applied Mathematics, Vol. 65.
Academic Press, New York-London, 1975.

\bibitem{Ca}
H. Cabannes,
{\em Theoretical Magnetofluiddynamics}, Academic Press, New York,
1970.

\bibitem{gw} G.-Q. Chen, D. Wang, \emph{ Global solution of nonlinear magnetohydrodynamics with large initial
data.} J. Differential Equations, 182 (2002), 344-376.

\bibitem{gw2} G.-Q. Chen, D. Wang,
{\em Existence and continuous dependence of large solutions for the
magnetohydrodynamic equations.}  Z. Angew. Math. Phys.  54  (2003),
608--632.

\bibitem{f3} B. Ducomet, E. Feireisl,  \emph{ The equations of Magnetohydrodynamics: On the interaction between matter
and radiation in the evolution of gaseous stars.} Commun. Math. Phys. 226 (2006), 595-629.

\bibitem{r1} R. Erban, \emph{ On the existence of solutions to the Navier-Stokes equations of a two-dimensional compressible flow.}
Math. Methods Appl. Sci. 26 (2003), 489--517.

\bibitem{FJN}
J. Fan, S. Jiang, and G. Nakamura,
{\em Vanishing shear viscosity limit in the magnetohydrodynamic equations},
Commun. Math. Phys. 270 (2007), 691-708.


\bibitem{f10} E. Feireisl, {\em Compressible Navier-Stokes equations with a non-monotone pressure law}, J. Diff. Equations 184 (2002), 97-108.

\bibitem{f2} E. Feireisl, \emph{ Dynamics of viscous compressible
fluids.}
Oxford Lecture Series in Mathematics and its Applications, 26.
Oxford University Press, Oxford, 2004.

\bibitem{f1} E. Feireisl, A. Novotn$\acute{\mathrm{y}}$ and H. Petzeltov$\acute{\mathrm{a}}$, \emph{ On the
existence of globally defined weak solutions to
the Navier-Stokes equations.} J. Math. Fluid Mech. 3 (2001),
358--392.

\bibitem{f5} E. Feireisl, H. Petzeltov\'{a}, \emph{ Large-time behaviour of solutions to the Navier-Stokes equations of
compressible flow.} Arch. Rational Mech. Anal. 150 (1999), 77-96.

\bibitem{h1} H. Freist\"{u}hler, P. Szmolyan, \emph{ Existence and bifurcation of viscous profiles for all intermediate
magnetohydrodynamic shock waves.} SIAM J. Math. Anal., 26 (1995), 112-128.

\bibitem{hoff95}
D. Hoff, {\em Strong convergence to global solutions for multidimensional flows of compressible, viscous fluids with polytropic equations of state and discontinuous initial data.}  Arch. Rational Mech. Anal.  132  (1995), 1--14.

\bibitem{hoff97}
D. Hoff, {\em Discontinuous solutions of the Navier-Stokes equations for multidimensional flows of heat-conducting fluids.}  Arch. Rational Mech. Anal.  139  (1997), 303--354.

\bibitem{HT}
D. Hoff, E. Tsyganov, {\em  Uniqueness and continuous dependence of weak solutions in compressible magnetohydrodynamics.} Z. Angew. Math. Phys. 56 (2005), 791--804.

\bibitem{KS}
V. Kazhikhov and V.~V. Shelukhin,
{\em Unique global solution with respect to time of
initial-boundary-value problems for one-dimensional equations of a
viscous gas},
{ J. Appl. Math. Mech.} 41 (1977), 273-282.

\bibitem{sm} S. Kawashima, M. Okada, \emph{ Smooth global solutions for the one-dimensional equations in
magnetohydrodynamics.} Proc. Japan Acad. Ser. A Math. Sci., 58 (1982),  384-387.

\bibitem{KL}
A.~G. Kulikovskiy and G.~A. Lyubimov,
{\em Magnetohydrodynamics},
Addison-Wesley, Reading, Massachusetts, 1965.

\bibitem{LL}
L.~D. Laudau and E.~M. Lifshitz,
{\em Electrodynamics of Continuous Media}, 2nd ed.,
Pergamon, New York, 1984.

\bibitem{p1} P. L. Lions, \emph{ Mathematical topics in fluid mechanics. Vol. 1. Incompressible
models.}
Oxford Lecture Series in Mathematics and its Applications, 3. Oxford
Science Publications. The Clarendon Press, Oxford University Press,
New York, 1996.

\bibitem{p2} P. L. Lions, \emph{ Mathematical topics in fluid mechanics. Vol. 2. Compressible
models.}
Oxford Lecture Series in Mathematics and its Applications, 10.
Oxford Science Publications. The Clarendon Press, Oxford University
Press, New York, 1998.

\bibitem{tz} T.-P. Liu, Y. Zeng, \emph{ Large time behavior of solutions for general quasilinear hyperbolic-
parabolic systems of conservation laws.} Memoirs Amer. Math. Soc. 599, 1997.

\bibitem{pm1} P. Maremonti, \emph{ Existence and stability of time-periodic solutions to the
Navier-Stokes equations in the whole space.} Nonlinearity 4
(1991), 503--529.

\bibitem{i1} A. Novotn\'{y}, I. Stra\v{s}kraba, {\em Introduction to the theory of compressible flow}, Oxford University Press: Oxford, 2004.

\bibitem{w1} D. Wang, \emph{ Large solutions to the initial-boundary value problem for planar magnetohydrodynamics.}
SIAM J. Appl. Math. 63 (2003), 1424-1441.

\bibitem{zy} Z. Yi, {\em An $L^p$ theorem for compensated compactness}, Proc. Royal Soc. Edinburgh A 122 (1992), 177-189.

\end{thebibliography}
\end{document}